\documentclass[10pt, journal, twocolumn]{IEEEtran} 

\usepackage{psfrag}
\usepackage{amsmath,amssymb,amsfonts,bbm}
\usepackage{latexsym}
\usepackage{graphicx}

\usepackage{colortbl}
\usepackage{fancyhdr}
\usepackage{epstopdf}
\usepackage{float}
\usepackage{hyperref}
\usepackage{booktabs}
\usepackage{enumerate}
\usepackage[ruled]{algorithm2e}
\usepackage{balance}
\usepackage{mathrsfs}

\usepackage{tikz}
\usepackage{pgfplots}
\pgfplotsset{compat=1.10}
\usepgfplotslibrary{fillbetween}

\newtheorem{theorem}{Theorem}
\newtheorem{definition}{Definition}
\newtheorem{proposition}{Proposition}
\newtheorem{lemma}{Lemma}

\newtheorem{remark}{Remark}
\newtheorem{assumption}{Assumption}

\IEEEoverridecommandlockouts

\newcommand{\bs}{\boldsymbol}
\newcommand{\mc}{\mathcal}
\newcommand{\bb}{\mathbb}
\newcommand{\R}{\bb R}

\newcommand{\norm}[1]{\left\|#1\right\|}

\newcommand{\dom}{\operatorname{dom}}

\newcommand{\blue}{\textcolor{black}}
\newcommand{\argmin}{\operatorname{argmin}}

\newcommand{\fix}{\mathrm{fix}}
\newcommand{\proj}{\mathrm{proj}}

\newcommand{\Id}{\mathrm{Id}}

\newcommand{\diag}{\operatorname{diag}}
\newcommand{\col}{\operatorname{col}}
\newcommand{\zer}{\operatorname{zer}}

\newcommand{\conv}{\mathrm{conv}}
\newcommand{\nc}{\mathrm{N}}
\newcommand{\0}{\mathbf{0}}
\newcommand{\1}{\mathbf{1}}

\newcommand{\Rmnum}[1]{\expandafter\@slowromancap\romannumeral #1@}

\makeatletter
\begin{document}

\title{Distributed generalized Nash equilibrium seeking in aggregative games on time-varying networks}

\author{Giuseppe Belgioioso, Angelia Nedi\'c and Sergio Grammatico
\thanks{G. Belgioioso is with the Control Systems group, TU Eindhoven, The Netherlands.
A. Nedi\'c is with the School of Electrical, Computer, and Energy Engineering, Arizona State University, USA.
S. Grammatico is with the Delft Center for Systems and Control (DCSC), TU Delft, The Netherlands. 
E-mail addresses: \texttt{g.belgioioso@tue.nl}, \texttt{angelia.nedich@asu.edu}, \texttt{s.grammatico@tudelft.nl}. 
This work was partially supported by NWO (projects OMEGA, 613.001.702; P2P-TALES, 647.003.003), the ERC (project COSMOS, 802348) and by the Office of Naval Research grant no. N000141612245.
}
}

\pagestyle{empty}
\maketitle         
\thispagestyle{empty}

\begin{abstract} 
We design the first fully-distributed algorithm for generalized Nash equilibrium seeking in aggregative games on a time-varying communication network, under partial-decision information, i.e., the agents have no direct access to the aggregate decision. The algorithm is derived by integrating dynamic tracking into a projected pseudo-gradient algorithm. The convergence analysis relies on the framework of monotone operator splitting and the Krasnosel'skii--Mann fixed-point iteration with errors.
\end{abstract}

\section{Introduction}
\IEEEPARstart{A}{n} aggregative game is a collection of inter-dependent optimization problems associated with noncooperative decision makers, or agents, where each agent is affected by some aggregate effect of all the agents \cite{jensen:10}. Remarkably, aggregative games arise in several applications, such as demand side management in the smart grid \cite{Saad2012}, e.g. for charging/discharging electric vehicles \cite{ma:callaway:hiskens:13}, demand-response regulation in competitive markets \cite{li2015demand}, congestion control in traffic and communication networks \cite{barrera:garcia:15}.
The common denominator is the presence of a large number of selfish agents, whose aggregate actions may disrupt the shared infrastructure, e.g. the power grid or the transportation network, if left uncontrolled.

\smallskip
Designing solution methods for multi-agent equilibrium problems in noncooperative games has recently gained high research interest. Several authors have developed semi-decentralized and distributed equilibrium seeking algorithms for games without coupling constraints \cite{ grammatico:parise:colombino:lygeros:16} and, more recently, for games with coupling constraints \cite{grammatico:17, paccagnan2018nash, belgioioso:grammatico:17cdc,yi2019operator}.

\smallskip
With focus on the generalized Nash equilibrium (GNE) problem, the formulations in \cite{belgioioso:grammatico:17cdc, yi2019operator} have introduced an elegant approach based on monotone operator theory \cite{bauschke2017convex} to characterize the equilibrium solutions as the zeros of a monotone operator. Not only is the monotone-operator-theoretic approach general -- e.g., unlike variational inequalities, smoothness of the cost functions is not required -- but also computationally viable, since several algorithmic methods to solve monotone inclusions are already well established, e.g.\ operator-splitting methods \cite[\S 26]{bauschke2017convex}. 

\smallskip
However, in the aforementioned literature on noncooperative equilibrium computation, it is assumed that the agents have direct access to the decisions of all their competitors, allowing every agent to evaluate its cost function without the need of extra communication. This game setup is known as \textit{full-decision information}.
\blue{
In aggregative games, this ideal scenario is achieved via the so-called semi-decentralized communication structure, where a central node gathers and broadcasts the aggregation variable to all the agents, see e.g. \cite{grammatico:17}-\cite{belgioioso:grammatico:17cdc}.
}

\smallskip
Recently, in the broader context of noncooperative games, the authors in \cite{tatarenko2018accelerated, pavel2019distributed} propose fully-distributed algorithms for equilibrium seeking under \textit{partial-decision information}, i.e., \blue{
each agent can only observe the decision of some neighboring agents, while its cost function possibly depends on all the other agents' decision.} 
In \cite{tatarenko2018accelerated}, to deal with the lack of information, the agents are endowed with auxiliary variables, namely, the estimates of the decisions of the other agents. Then, a consensus protocol is combined with accelerated projected-pseudo-gradient dynamics to steer the estimates towards their real value and, consequently, the decisions to a Nash equilibrium, in the same time-scale. In \cite{pavel2019distributed}, similar ideas are developed in the general framework of monotone operator theory to design an algorithm for games with coupling constraints.
The algorithms proposed in \cite{tatarenko2018accelerated, pavel2019distributed} require a number of auxiliary variables (i.e., the estimates of the decisions of all the other agents) which is proportional to the number of agents in the game. From a practical perspective, this can be regarded as a drawback in terms of memory storage and communication requirements, especially in games with very large number of agents.

\smallskip
Scalability with respect to the population size indeed motivates us to focus on aggregative games.
 In this context, the authors in \cite{koshal:nedic:shanbhag:16} propose an algorithm that relies on dynamic tracking, \blue{a technique
that allows a group of agents to locally track the average of some reference inputs}, extensively used in distributed optimization for gradient tracking, e.g.  \cite{nedic2017achieving}. Specifically, the authors embed dynamic tracking of the aggregate decision in a projected-pseudo-gradient update to compute a Nash equilibrium in a fully-distributed fashion (i.e., without the need of a central coordinator).
In the context of aggregative games with coupling constraints, an algorithm is proposed in \cite{parise2019distributed}, however with important limitations: it requires a very large number of distributed communication rounds before each strategy update; convergence is guaranteed to approximate solutions (i.e., $\varepsilon-$Nash equilibria) only; the communication network must be time-invariant. 

\blue{More recently, two fully-distributed algorithms \cite{gadjov2019distributed, belgioioso2019distributed}, for generalized aggregative games over time-invariant and connected networks, have been proposed to compute an exact solution (i.e., GNE), without the need of multiple communication rounds before every strategy update.
To cope with the lack of information, both algorithms introduce local estimates and dynamic tracking of the aggregate decision. In \cite{gadjov2019distributed}, global convergence is proved under strong monotonicity of the pseudo-gradient, by leveraging a rescrited-monotonicity property of this mapping in the exteneded space of strategies and estimates. In our preliminary work \cite{belgioioso2019distributed}, this assumption is relaxed to cocoercivity at the cost of having vanishing step-sizes, which typically imply slow convergence.
Unfortunately, the extension of both methodologies to cover time-varying communication networks is currently missing, since the operator theoretic framework on the basis of their convergence analysis fails when the underlying mappings vary over time.
}

\smallskip
\subsubsection*{Contribution}
In this paper, we solve these technical issues and propose the first discrete-time, fully-distributed algorithm to compute a generalized Nash equilibrium in aggregative games with coupling constraints over a time-varying and repeatedly-connected	 communication network. The algorithm is obtained by combining dynamic tracking, projected-pseudo-gradient and Krasnosel'skii--Mann dynamics. The key approach to prove convergence of our proposed algorithm relies on applying and tailoring the framework of operator splitting methods \cite{bauschke2017convex} and fixed-point iteration \textit{with errors} \cite{combettes2001quasi}.

\smallskip
\subsubsection*{Organization of the paper}
In Section \ref{sec:PS}, we formalize the generalized Nash equilibrium seeking problem for aggregative games over a time-varying communication network. In Section \ref{sec:DAlg}, we present a fully-distributed algorithm and discuss its interpretation from an operator theoretic and fixed-point perspective. In Section \ref{sec:ConvA}, we establish global convergence of the proposed method. To corroborate the theory, in Section \ref{sec:NS}, we study the performance of the proposed method on a Nash--Cournot game. Concluding remarks and future research directions are discussed in Section \ref{sec:Concl}.

\medskip
\subsubsection*{Basic notation}
$\R$ denotes the set of real numbers, and $\overline{\R} := \R \cup \{\infty\}$ the set of extended real numbers. $\bs{0}$ ($\bs{1}$) denotes a matrix/vector with all elements equal to $0$ ($1$); to improve clarity, we may add the dimension of these matrices/vectors as subscript. \blue{Given two sets, $\mc S_1$ and $\mc S_2$, we denote as $\mc S_1 \times \mc S_2$ their Cartesian product.
Given $N$ sets, $\mc S_1, \ldots, \mc S_N$, we denote with
$\conv(\mc S_1, \ldots, \mc S_N) = \big\{ a_1 x_1 + \ldots + a_N x_N \, | \;  \sum_{i=1}^N a_i = 1, \, a_i \in \R_{\geq 0}, \, x_i \in \mc S_i, \, \forall i \in \{1,\ldots,N \} \big\}$ the convex hull of their union.} $A \otimes B$ denotes the Kronecker product between the matrices $A$ and $B$. For a square matrix $A=[a_{i,j}] \in \R^{n \times n}$, where $a_{i,j}$ is the entry in position $(i,j)$, its transpose is $A^\top$; $A \succ 0$ ($\succeq 0$) stands for positive definite (semidefinite) matrix; \blue{$\left\| A \right\|$ denotes the largest singular value of $A$; $\left\| A \right\|_{\infty} = \max_{1 \leq i \leq n} \sum_{i =1}^n | a_{i,j}|$ denotes the infinity norm. If $A \succ 0$, $\|\cdot\|_A $ denotes the $A$-induced norm, such that $\| x\|_A=\sqrt{x^\top A x}$, we omit the subscript when $A = I$.} Given $N$ matrices $A_1, \ldots, A_N$, $\textrm{blkdiag}(A_1, \ldots, A_N)$ denotes a block diagonal matrix with $A_1, \ldots, A_N$ as diagonal blocks. Given $N$ vectors $x_1, \ldots, x_N$, $\boldsymbol{x} := \col\left(x_1,\ldots,x_N\right) = [ x_1^\top, \ldots , x_N^\top ]^\top$, $\bar{x} = \frac{1}{N}\sum_{i=1}^N x_i$, $\bs x_{-i} := \col(x_1, \ldots,x_{i-1},x_{i+1},\ldots, x_N)$; given a vector $z$, 
$(z,\bs x_{-i}) := \col(x_1, \ldots,x_{i-1},z,x_{i+1},\ldots, x_N)$.

\smallskip
\subsubsection*{Operator theoretic definitions}
$\Id(\cdot)$ denotes the identity operator. The mapping $\iota_{S}:\R^n \rightarrow \{ 0, \, \infty \}$ denotes the indicator function for the set $S \subseteq \R^n$, i.e., $\iota_{S}(x) = 0$ if $x \in S$, $\infty$ otherwise. For a closed set $S \subseteq \R^n$, the mapping $\proj_{S}:\R^n \rightarrow S$ denotes the projection onto $S$, i.e., $\proj_{S}(x) = \argmin_{y \in S} \left\| y - x\right\|$. The set-valued mapping $\nc_{S}: \R^n \rightrightarrows \R^n$ denotes the normal cone operator for the set $S \subseteq \R^n$, i.e., 
$\nc_{S}(x) = \varnothing$ if $x \notin S$, $\left\{ v \in \R^n \mid \sup_{z \in S} \, v^\top (z-x) \leq 0  \right\}$ otherwise.
For a function $\psi: \R^n \rightarrow \overline{\R}$, $\dom(\psi) := \{x \in \R^n \mid \psi(x) < \infty\}$; $\partial \psi: \dom(\psi) \rightrightarrows {\R}^n$ denotes its subdifferential set-valued mapping, defined as $\partial \psi(x) := \{ v \in \R^n \mid \psi(z) \geq \psi(x) + v^\top (z-x)  \textup{ for all } z \in {\rm dom}(\psi) \}$.
A set-valued mapping $\mathcal{F} : \R^n \rightrightarrows \R^n$ is
(strictly) monotone if $(u-v)^\top  (x-y) \geq (>) \, 0$ for all $x \neq y \in \R^n$, $u \in \mathcal{F} (x)$, $v \in \mathcal{F} (y)$; 
$\mc F$ is restricted-(strictly) monotone with respect to (w.r.t.) $Y \subset \R^n$ if $(z^*-z)^\top  (x^*-x) \geq (>) 0$
for all $\forall \bs x^* \in Y$, $\bs x \in \R^n \setminus Y$, $\bs z^* \in \mc F(x^*)$, $\bs x \in \mc F(x)$;
$\mathcal{F} $ is $\eta-$strongly monotone, with $\eta>0$, if 
$(u-v)^\top (x-y) \geq \eta \left\| x-y \right\|^2$ for all $x \neq y \in \R^n$, $u \in \mathcal{F} (x)$, $v \in \mathcal{F} (y)$;
$\fix\left( \mathcal{F}\right) := \left\{ x \in \R^n \mid x \in \mathcal{F}(x) \right\}$ and $\zer\left( \mathcal{F}\right) := \left\{ x \in \R^n \mid 0 \in \mathcal{F}(x) \right\}$ denote the set of fixed points and of zeros, respectively. A single-valued mapping $F: \R^n \rightarrow \R^n$ is $L$-Lipschitz continuous, with $L>0$, if $\|F(x)-F(y)\| \leq L \|x-y\|$ for all $x,y \in \R^n$; $F$ is nonexpansive if it is $1$-Lipschitz continuous;
$F$ is $\eta$-{averaged}, with $\eta \in (0,1)$, if  
$\left\| F (x) - F(y) \right\|^2 \leq \left\| x-y \right\|^2 - \tfrac{1-\eta}{\eta}\left\| \left( \textup{Id}-F \right)(x) - \left(\textup{Id}- F  \right)(y) \right\|^2$, for all $x, y \in \R^n$;
$F $ is $\beta$-cocoercive, with $\beta>0$, if $\beta F $ is $\tfrac{1}{2}$-averaged.

\section{Problem statement}
\label{sec:PS}
Consider a set of $N$ agents indexed by $\mc I = \{1,\ldots ,N \}$. The $i$-th
agent is characterized by a local strategy set $\Omega_i \subset \R^n$ and a cost function $J_i(x_i, \bar{x})$, which depends on the decision of agent $i$, $x_i$, and on the aggregate of all agent decisions, i.e.,
\begin{equation*}
\bar{x} := \tfrac{1}{N} \textstyle \sum_{i=1}^N x_j. 
\end{equation*}
Moreover, we assume that the collective strategy profile $\bs x := \col(x_1, \ldots, x_N) \in \R^{nN}$ must satisfy a coupling constraint, described by the affine function $\bs x \mapsto C \bs x -c $, where $C = [C_1|\ldots|C_N]\in \R^{m\times nN}$, $c = \sum_{i =1}^N c_i \in \R^m$, and $C_i$, $c_i$ are local parameters known to agent $i$ only. In summary, the aim of each agent $i$, given the decision variables of the other agents, i.e., $\bs x_{-i} := \col(x_1, \ldots,x_{i-1},x_{i+1},\ldots, x_N)$, is to choose a strategy $x_i$ that solves its local optimization problem, according to the game setup above, i.e., $\forall i \in \mc I :$
\begin{align} \textstyle \label{eq:Game}
\left\{
\begin{array}{c l}
\underset{x_i \in \R^n}{\argmin} &
J_i(x_i, \frac{1}{N}x_i + \frac{1}{N} \sum_{j\neq i} x_j)\\
\text{s.t.} & x_i \in \Omega_i \\[0.2em]
& C_i x_i - c_i \leq \sum_{j \neq i}^N (c_j - C_j x_j)
\end{array}
\right.
\end{align}
where the last constraint is equivalent to $C \bs x -c \leq \0$.
\blue{
\smallskip
\begin{remark}
Affine coupling constraints, as considered in this paper, are very common in the literature of noncooperative games, e.g.
\cite{paccagnan2018nash, yi2019operator, pavel2019distributed, parise2019distributed}, and cover several applications where they typically arise in the form of upper and lower limits on the available shared resources, e.g. \cite{Saad2012}-\cite{barrera:garcia:15}.
{\hfill $\square$}
\end{remark}
}

\smallskip
\begin{assumption} \label{ass:ConvCompSet}
For all $i \in \mc I$ and any fixed $u \in \frac{1}{N}\sum_{j \neq i}^N \Omega_j$, the function $J_i(\cdot \, , \frac{1}{N} \cdot +\, u)$ is convex and continuously differentiable, $\Omega_i \subset \R^n$ is non-empty, compact and convex. The global feasible set $K := \{\bs x \in \prod_{i =1}^N \Omega_i |\,  C \bs x - c \leq \0 \}$ is non-empty and satisfies Slater's constraint qualification.
{\hfill $\square$}
\end{assumption}

\smallskip
From a game-theoretic perspective, our goal is to distributively compute a generalized Nash equilibrium of the aggregative game described by the $N$ inter-dependent optimization problems in \eqref{eq:Game}.

\smallskip
\begin{definition}[Generalized Nash equilibrium]
A collective strategy $\bs x^{*} \in K$ is a generalized Nash equilibrium (GNE) of the game in \eqref{eq:Game} if%
, for all $ i\in \mc I$:
\begin{multline*} \textstyle 
J_i\left( x^*_{i}, \bar{x}^* \right) \leq 
\ J_{i} \left( z, \, \frac{1}{N}z+ \frac{1}{N}\sum_{j\neq i}^N x_j^* \right), \\[.2em]
\forall z \text{ s.t. } (z, \bs{x}_{-i}^*)\in K.
\end{multline*}
%
\end{definition}

\subsection{Communication networks}
\label{sec:com_network}
We consider a time-varying network to model the communications among agents over time.
At each stage $k$, the communication is described by an undirected graph $\mc G_k =(\mc I,\mc E_k)$, where $\mc I $ is the set of vertices (agents) and $\mc E_k \subseteq \mc I \times \mc I$ is the set of edges.
An unordered pair of vertices $(i,j)$ belongs to $ \mc E_k$ if and only if agents $j$ and $i$ can exchange information.
The set of neighbors of agent $i$ at stage $k$ is defined as $ \mc N_i(k) = \{ j | \, (i,j)\in \mc E_k \}$.
Next, we assume the graphs sequence $\{\mc G_k\}_{k \in \bb N}$ to be $Q-$connected.

\smallskip
\begin{assumption} \label{ass:ConnGra}
There exists an integer $Q \geq 1$ such that the
graph $(\mc I, \cup_{\ell=1}^Q \mc E_{\ell+k})$ is connected, for all $k \geq 0$.
{\hfill $\square$}
\end{assumption}

\smallskip
This assumption ensures that the intercommunication
intervals are bounded for agents that communicate directly. In other words, every agent sends information to each of its neighboring agents at least once every $Q$ time intervals.

\smallskip
We consider a mixing matrix $W(k) = [w_{i,j}(k)]$ associated with $\mc G_k$, whose elements satisfy the following assumption.

\begin{assumption} \label{ass:DSMM}
For all $k \in \bb N$, the matrix $W(k)=[w_{i,j}(k)] $ satisfies the following conditions:
\begin{enumerate}[(i)]
\item (Edge utilization) Let $i,j \in \mc I$, $i \neq j$. If $(i,j) \in \mc E_k$, $w_{i,j}(k) \geq \epsilon$, for some $\epsilon > 0$; $w_{i,j}(k) = 0$ otherwise;
\item (Positive diagonal) For all $i \in \mc I$, $w_{i,i}(k) > \epsilon$;
\item (Double-stochasticity) $W(k) \1 = \1$, $\1^\top W(k) = \1^\top$.
{\hfill $\square$}
\end{enumerate}
\end{assumption}

\smallskip
Assumption \ref{ass:DSMM} is strong but typical for multiagent coordination and optimization, e.g. \cite{nedic2017achieving, margellos2018distributed}. For an undirected graph it can be fulfilled, for example, by using Metropolis weights:
\begin{equation} \label{eq:Metropolis}
w_{i,j}(k)=
\begin{cases}
(\max\{|\mc N_i(k)|, |\mc N_j(k)| \})^{-1} & \text{if } (i,j) \in \mc E_k,\\
0 & \text{if } (i,j) \not\in \mc E_k, 
\\
1- \sum_{\ell \in \mc N_i} w_{i,\ell}(k) & \text{if } i=j.
\end{cases}
\end{equation}



Finally, let us introduce the so-called transition matrices $\Psi(k,s)$ from time $s$ to $k$:
\begin{equation} \label{eq:TransMatr}
\Psi(k,s) = W(k)W(k-1)\cdots W(s+1)W(s),
\end{equation}
for $0 \leq s <k$, where $\Psi(k,k)=W(k)$, for all $k$. The following statement shows the convergence properties of the transition matrix $\Psi(k,s)$.

\begin{lemma}[{\cite[Lemma 5.3.1]{tsitsiklis1984problems}}] \label{lem:Tsi}
Let Assumptions \ref{ass:ConnGra}, \ref{ass:DSMM} hold true. Then, the following statements hold:
\begin{enumerate}[(i)]
\item $\lim_{k \rightarrow \infty} \Psi(k,s) =(1/N)\1 \1^\top$, for all $s \geq 0$.
\item The convergence rate of $\Psi(k,s)$ is geometric, i.e., $\|\Psi(k,s) -(1/N)\1 \1^\top \| \leq \theta \rho^{k-s}$ for all $k \geq s \geq 0$, where $\theta := N(1- \epsilon/(4N^2))^{-2}$ and 
\begin{equation} \label{eq:rho}
\textstyle
\rho := (1-
\frac{\epsilon}{4N^2})^{1/Q} \in (0,1),
\end{equation}
with $Q$ as in Assumption \ref{ass:ConnGra} and $\epsilon$ as in Assumption \ref{ass:DSMM}.
{\hfill $\square$}
\end{enumerate}
\end{lemma}

\subsection{GNE as zeros of a monotone operator}
As first step, we characterize a GNE of the game in terms of the KKT conditions of the coupled optimization problems in \eqref{eq:Game}. For each agent $i \in \mathcal{I}$, let us introduce the Lagrangian function $L_i$, defined as 
\begin{equation*}
L_i(\bs x,\lambda_i) := J_i(x_i,\bar{x})+ \iota_{\Omega_i}(x_i) +\lambda_i^\top (C \bs x-c),
\end{equation*}
where $\lambda_i \in \R^m_{\geq 0}$ is the dual variable of agent $i$ associated with the coupling constraints, and $\iota_{\Omega_i}$ is the indicator function.
It follows from \cite[\S 12.2.3]{palomar2010convex} that the set of strategies $\bs x^*$ is a GNE of the game in \eqref{eq:Game} if and only if the following coupled KKT conditions are satisfied for some $\lambda_1,\ldots,\lambda_N \in \R^m_{\geq 0}$:
\begin{equation} \label{eq:KKT}
\forall i \in \mc I:
\begin{cases}
0 \in \nabla_{x_i} J_i(x_i^*,\bar{x}^*) + \nc_{\Omega_i}({x}^*_i) + C_i^\top \lambda_i^*,\\
0 \leq \lambda_i^* \perp -(C {\bs x}^*-c) \geq 0.
\end{cases} 
\end{equation}

\blue{
Within all the possible GNE, we focus on
an important subclass of equilibria, namely
the variational GNE (v-GNE), that
enjoy some relevant structural properties, such as ``larger social stability” and ``economic fairness”  and corresponds to the solution set of the KKT conditions in \eqref{eq:KKT} with equal dual variables, i.e., $\lambda_1^* = \ldots = \lambda_N^*$ \cite[Theorem 3.1]{facchinei:fischer:piccialli:07}.} The next proposition characterizes the subclass of v-GNE as the
solution to a specific variational inequality problem%
\footnote{For a single-valued mapping $M:\R^n \rightarrow \R^n$ and a set $\mc S \subseteq \R^n$, the variational inequality problem VI$(M, \mc S )$ is the problem of finding a vector $ \omega^* \in \mc S $ such that $ M(\omega^*)^\top (\omega - \omega^*) \geq 0$, for all $ \omega  \in \mc S $, \cite[Def. 1.1.1]{facchinei:pang}.
}, 
or equivalently as the zero set of the set-valued mapping 
\begin{align} \label{eq:U}
U: \begin{bmatrix}
\bs x \\  \lambda
\end{bmatrix} \mapsto
 \begin{bmatrix}
\nc_{\bs \Omega}(\bs x)+ F(\bs x ) + C^\top \lambda \\
\nc_{\R^{m}_{\geq 0}}( \lambda)  - (C \bs x - c)
\end{bmatrix},
\end{align}
where $\lambda \in \bb R^m$, $\bs \Omega:= \prod_{i=1}^N \Omega_i$, $\nc_{\mc S} = \partial \iota_{\mc S}$ is the normal cone operator associated with a set $\mc S$ and $F$ is the so-called pseudo-gradient mapping (PG) defined as
\begin{align} \label{eq:F}
F(\bs x) = \col( \nabla_{x_1}J_1(x_1,\bar x), \ldots,\nabla_{x_N} J_N(x_N,\bar x)).
\end{align}

\smallskip
\begin{proposition} \label{pr:UvGNE}
Let Assumption \ref{ass:ConvCompSet} hold.
Then, the following statements are equivalent:
\begin{enumerate}[(i)]
\item $\bs x^*$ is a variational GNE of the game in \eqref{eq:Game};
\item $\exists \lambda^* \in \R^m_{\geq 0}$ such that, the pair $(x_i^*,\lambda^*)$ is a solution to the KKT in \eqref{eq:KKT}, for all $i \in \mc I$; 
\item $\bs x^*$ is a solution to VI$(F,K)$;
\item $\exists \lambda^* \in \R^m_{\geq 0}$ such that $\col(\bs x^*, \lambda^*) \in \zer (U)$.
{\hfill $\square$}
\end{enumerate}
\end{proposition}

\smallskip
\begin{proof}
The equivalences (i)$\Leftrightarrow$(ii)$\Leftrightarrow$(iii) are proven in \cite[Th. 3.1]{facchinei:fischer:piccialli:07} while (iii)$\Leftrightarrow$(iv) follows by \cite[Th. 3.1]{auslender2000lagrangian}.
\end{proof}

\smallskip
The following assumptions on the PG in \eqref{eq:F} are standard (e.g. \cite[Th. 3]{paccagnan2018nash}, \cite[Assumption 2]{yi2019operator}, \cite[Assumption 3]{belgioioso2018projected}) and sufficient to ensure the convergence of standard GNE seeking algorithms based on projected-pseudo-gradient dynamics.

\smallskip
\smallskip
\begin{assumption} \label{ass:COCO}
$F$ in \eqref{eq:F} is $\chi-$cocoercive over $\bs \Omega$.
{\hfill $\square$}
\end{assumption}

\smallskip
When $F$ is $\xi-$strongly monotone and $L_{\text{F}}-$Lipschitz, then $F$ is also $(\xi/L_{\text{F}}^2)-$cocoercive. However, in general, cocoercive mappings are not necessarily strongly monotone, e.g. the gradient of a (non-strictly) convex \blue{and smooth} function.

\blue{
To emphasize the structure of $F$ in \eqref{eq:F}, we define
\begin{align} \label{eq:tildeFcomp} \textstyle
F_i(v, w) := 
\left.
\left( \frac{\partial}{\partial z_1} J_i (z_1, z_2 ) + \frac{1}{N} \frac{\partial}{\partial z_2} J_i (z_1, z_2 ) \right) \right|_{
\begin{smallmatrix} z_1 = v \\ \, z_2 = w
\end{smallmatrix}
} ,
\end{align}
that satisfies
$
F_i(x_i, \bar x) = \nabla_{x_i}J_1(x_i,\bar x)
$, for all $ i \in \mc I$.
Then, we define the extended pseudo-gradient mapping (EPG) 
\begin{align} \label{eq:tildeF}
\bs F(\bs v, \bs w) := \col \big(  F_1(v_1, w_1),\ldots,  F_N(v_N, w_N) \big),
\end{align}
where each component mapping $F_i$ is given by \eqref{eq:tildeFcomp}. With this notation, we have $\bs F(\bs x,  \1 \otimes \bar x) = F(\bs x)$.}
Next, we assume Lipschitz continuity of the EPG, which is usual in the context of games under partial-decision information, see e.g. \cite[Assumption 3]{pavel2019distributed}, \cite[Assumption 3]{koshal:nedic:shanbhag:16}, \cite[Assumption 4]{gadjov2019distributed}.
 
 \smallskip
\begin{assumption} \label{ass:LC-TF}
\blue{Let $\bar \Omega:= \conv \left(\Omega_1, \ldots, \Omega_N \right)$ be the set whose elements are convex combination of the elements from the local sets $\Omega_i$'s.} The mapping $\bs F$ in \eqref{eq:tildeF} is uniformly Lipschitz continuous over $\bs \Omega \times \bar{ \bs \Omega}$, \blue{with $\bar{\bs \Omega} = \prod_{i=1}^N \bar \Omega$}, i.e., there exists $L_{\bs F}>0$ such that, for all $\bs v, \bs u \in \bs  \Omega$ and $\bs w, \bs z \in \bar{ \bs \Omega}$,
$$
\| \bs F(\bs v,\bs w)- \bs F(\bs u, \bs z)\| \leq L_{\bs F} \, \|
\left[
 \begin{smallmatrix}
 \bs v \\ \bs w
 \end{smallmatrix}
 \right]
 -
 \left[
 \begin{smallmatrix}
 \bs u \\ \bs z
 \end{smallmatrix}
\right] 
 \|.$$

\vspace*{-1.8em} {
\hfill $\square$}
\end{assumption}

\medskip
\begin{remark}[Existence and uniqueness of a v-GNE] \label{rem:REU}
It follows by \cite[Cor.~2.2.5]{facchinei2007finite} that VI($F$,$K$) has a non-empty and compact solution set, since $K$ is non-empty, compact and convex and $F$ is continuous, by Assumption \ref{ass:ConvCompSet}. Furthermore,  when $F$ is strictly monotone, then the solution to VI($F$,$K$), (i.e., the v-GNE of the game), is unique \cite[Th.~2.3.3]{facchinei2007finite}. 
{\hfill $\square$}
\end{remark}

\subsection{Boundedness of the dual variables}
\label{subsec:BDV}
In the next statement, we formally establish the boundedness of the dual solution set of VI($F,K$) or, equivalently, of the dual part of the monotone inclusion $\col(\bs x^*, \lambda^*) \in \zer (U) $. 
\smallskip
\begin{lemma} \label{lem:DualBoundedness}
Let Assumptions \ref{ass:ConvCompSet} hold true. If $\col(\bs x^*,\lambda^*) \in \zer(U)$, then $\lambda^* \in D^*$, where $D^* \subset \R^m_{\geq 0}$ is bounded.%
{\hfill $\square$}
\end{lemma}
\begin{proof}
The boundedness of the dual solution set $D^*$ follows by \cite[Proposition 3.3]{auslender2000lagrangian} since VI($F$, $K$) has a non-empty bounded solution set by Remark \ref{rem:REU} and there exists a vector $\bs x \in \dom( F )$ satisfying Slater's constraint qualification by Assumption \ref{ass:ConvCompSet}.
\end{proof}

\smallskip
\blue{
Let us denote with $B_{D^*} = \max_{\lambda \in D^*} \| \lambda \|_{\infty}$ the largest entry of all the optimal dual vectors. The agents can locally build a bounded superset $D_\text{i}$ of the optimal dual set $D^*$ as follows: $D_\text{i} := \{ \mu \in \R^{m}_{\geq 0} \, | \; \| \mu \|_{\infty} \leq B_{D^*} + r, \; \text{ with } r > 0 \}$ \cite[p. 21]{nedic2009approximate}.
In the context of distributed constrained optimization, a local estimate of $B_{D^*}$ can be constructed based on a Slater's vector, see \cite[\S 4.2]{nedic2009subgradient}, \cite[\S 3.A (2)]{zhu2012distributed}. The extension of these estimation methods to generalized noncooperative games would rely on Lagrangian duality theory for variational inequalities \cite{auslender2000lagrangian}.
In practice, each agent does not need an accurate estimate of the optimal dual solution set $D^*$ and can simply construct a local superset $D_\text{i}$ by taking $r$ large enough.
}

\subsection{A standard semi-decentralized algorithm}
It follows by Proposition \ref{pr:UvGNE} that the original GNE seeking problem corresponds to the following monotone inclusion problem:
\begin{equation} \label{eq:MIP}
\text{find } \bs \omega^*=\col(\bs x^*, \lambda^*) \text{ s.t. } \0 \in U(\omega^*).
\end{equation}

Next, we recall a standard semi-decentralized GNE seeking algorithm obtained by solving the monotone inclusion problem in \eqref{eq:MIP} by means of a preconditioned forward-backward (pFB) splitting \cite[Alg. 1]{belgioioso2018projected}.

\smallskip
\begin{center}
\begin{minipage}{\columnwidth}
\hrule
\smallskip
\textsc{Algorithm 1.} \blue{Semi-decentralized v-GNE seeking}
\smallskip
\hrule
\medskip

Iterate until convergence
\begin{align*}
\begin{array}{l}
\text{In parallel, for all } i \in \mc I:\\[.2em]
\quad \left| 
\begin{array}{l}
x_i^{k+1} = \proj_{\Omega_i}\big( x_i^k - \alpha_i( F_i(x_i^k,\bar{x}^k)+ C_i^\top \lambda^k)\big) \\[.3em]
d_i^{k+1} = 2 C_i  x_i^{k+1} - C_i x_i^k - c_i\\[.5em]
\end{array}
\right. \\[1em]
\text{Central coordinator:}\\[.2em]
\quad \left| \; \,
\lambda^{k+1} = \proj_{\R^m_{\geq 0}} \big( \lambda^k + \beta N \bar{d}^{\, k+1} \big)
\right.
\end{array}
\end{align*}
\hrule
\end{minipage}
\end{center}

\medskip
\begin{remark} The local auxiliary variables $d_i$'s are introduced to cast Algorithm 1 in a more compact form. The average $\bar d^{\, k+1} := \frac{1}{N} \sum_{i=1}^N (2 C_i x_i^{k+1} - C_i x_i^k - c_i) $ measures the violation of the coupling constraints\blue{, technically, is the ``reflected violation" of the constraints at iteration $k$.}
{\hfill $\square$}
\end{remark}

\smallskip
If the step sizes $\{\alpha_i \}_{i \in \mc I}$ and $\beta$ are chosen small enough, then the sequence $( \col(\bs x^k, \lambda^k ))_{k \in \bb N}$ generated by Algorithm 1 converges to some $\col(\bs x^*, \lambda^*) \in \zer(U)$, where $\bs x^*$ is a v-GNE, see \cite[Th. 1]{belgioioso2018projected} for a formal proof of convergence.

\smallskip
We note that Algorithm 1 is not distributed. In fact, at each iteration $k$, a central coordinator is needed to:
\begin{enumerate}[(i)]
\item gather and broadcast the average strategy $\bar{x}^k$;
\item gather the average quantity $\bar d^k$; 
\item update and broadcast the dual variable $ \lambda^k$.
\end{enumerate}

\section{A distributed GNE seeking algorithm}
\label{sec:DAlg}
\subsection{Towards a fully distributed algorithm}
A first step towards a fully-distributed algorithm consists of endowing each agent with a copy, $\lambda_i$, of the dual variable and enforcing consensus on the local copies. Consider the set-valued mapping $T$, obtained by augmenting $U$ in \eqref{eq:U} with the local copies of the dual variable:
\begin{align} \label{eq:T}
T: \begin{bmatrix}
\bs x \\ \bs \lambda 
\end{bmatrix} \mapsto
 \begin{bmatrix}
 \nc_{\bs \Omega}(\bs x)+ F(\bs x ) + \frac{1}{N} C_{\text f}^\top \bs \lambda \\
\nc_{\R^{mN}}(\bs \lambda) 
+  \L_m \bs \lambda - \frac{1}{N}  ( C_{\text f}\, \bs x - c_{\text f})
\end{bmatrix},
\end{align}
where $\bs \lambda = \col(\lambda_1,\ldots,\lambda_N)$, $C_{\text f} = \1_N \otimes C$, $c_{\text{f}} = \1 \otimes c$, $\L_m = \L \otimes I_m$ and $\L:=I_{N} - \frac{1}{N} \1 \1^\top $ represents the projection onto the disagreement space. 

\smallskip
\begin{remark}
When the local copies of the dual variable are equal, i.e., $\bs \lambda \in \bs{E}^{\parallel} := \{  \1_N \otimes \lambda, \, | \, \lambda \in  \R^m \} $, where $\bs{E}^{\parallel}$ is the consensus subspace of dimension $m$, the first row block of $T$ corresponds to that of $U$, while each of the $N$ components of the second row block of $T$ describes the same complementarity condition, namely, the second row block of $U$.
{\hfill $\square$}
\end{remark}

We note that the mapping $T$ in \eqref{eq:T} can
be written as the sum of two operators, i.e.,
\begin{align} \label{eq:T1}
T_1&: \textstyle
\col(\bs x, \bs \lambda) \mapsto \col(F(\bs x),\L_m \bs \lambda + \frac{1}{N} c_{\text{f}} ),\\
\label{eq:T2}
T_2&:\col(\bs x, \bs \lambda) \mapsto \nc_{\bs \Omega}(\bs x) \times  \nc_{\R^{mN}_{\geq 0}}(\bs \lambda)+ S \col(\bs x, \bs \lambda) ,
\end{align}
where $S$ is a skew-symmetric linear mapping defined as
\begin{align} \label{eq:S}
S := \frac{1}{N} \begin{bmatrix}
 0 & C_{\text{f}}^\top \\
 -C_{\text{f}} & 0
 \end{bmatrix}. 
\end{align}
The formulation $T = T_1 + T_2$ is called splitting of $T$, and will be exploited in different ways later on.
The next lemma shows that $T_2$ is maximally monotone and that $T_1$ is cocoercive and strictly monotone with respect to the consensus subspace of the dual variables, i.e., $\bs \Omega \times \bs{E}^{\parallel}$.

\smallskip
\begin{lemma}\label{lem:RSM} 
Let Assumptions \ref{ass:ConvCompSet}, \ref{ass:COCO} hold true. The following statements hold:
\begin{enumerate}[(i)]
\item $T_2$ in \eqref{eq:T2} is maximally monotone on $\bs \Omega \times \R^{mN}_{\geq 0}$;
\item $T_1$ in \eqref{eq:T1} is $\delta-$cocoercive, with $0 \!< \!\delta \!\leq \! \min \{ 1,\chi \}$ and \textit{restricted-strictly monotone} w.r.t. $ \Theta^\parallel:=\bs \Omega \times \bs{E}^{\parallel} $, i.e., for all $\bs \omega^{\parallel} \in \Theta^\parallel$, $\bs \omega \in (\bs \Omega \times \R^{mN}_{\geq 0}) \setminus
 \Theta_\parallel$, it holds that $( T_1(\bs \omega) - T_1(\bs \omega^\parallel) )^\top ( \bs \omega-\bs \omega^\parallel ) > 0$;

\item $T$ is maximally monotone on $\bs \Omega \times \R^{mN}_{\geq 0}$ and \textit{restricted-strictly monotone} w.r.t. $ \Theta^\parallel$.
{\hfill $\square$}
\end{enumerate}
\end{lemma}
\begin{proof}
See Appendix \ref{proof:lemRSM}.
\end{proof}

\smallskip
The next proposition exploits the restricted-strict monotonicity of $T$ to shows that the v-GNE of the original game are fully characterized by the zeros of $T$.

\smallskip
\begin{proposition}\label{pr:EOC}
Let Assumption \ref{ass:ConvCompSet} hold true. The following statements hold:
\begin{enumerate}[(i)]
\item $\zer(T)\neq \varnothing$,
\item If $\col(\bs x^*, \bs \lambda^*) \in \zer(T)$, then $\bs x^*$ is a v-GNE and $\bs \lambda^* = \col( \lambda^*, \ldots, \lambda^*)$, with $\lambda^* \in \R^m_{\geq 0}$. 
{\hfill $\square$}
\end{enumerate}
\end{proposition}
\begin{proof}
See Appendix \ref{prof:prEOC}.
\end{proof}

\smallskip
To find a zero of $T$, we exploit a preconditioned version of the forward-backward method \cite[\S 25.6]{bauschke2017convex} on the splitting \eqref{eq:T1}-\eqref{eq:T2}, similarly to \cite{yi2019operator, belgioioso2018projected}, thus obtaining Algorithm 2.

The next theorem establishes global convergence of Algorithm 2 to a v-GNE if the step-sizes are chosen according to the following choices.

\begin{assumption} \label{ass:SS-choice1}
Take $ 0 < \delta \leq \min \{ 1, \chi \}$, where $\chi$ as in Assumption \ref{ass:COCO}. Set the global parameter $\tau > \frac{1}{2\delta} $ and denote $\nu:= \frac{2\delta\tau}{4\delta\tau -1} \in (1/2,1)$.
Set the step-sizes as follows:
\begin{enumerate}[(i)]
\item $0<\alpha_i \le (\|C_i\|+\tau )^{-1}$, for all $i \in \mc I$,

\smallskip
\item $0<\beta_i \le (\frac{1}{N}\sum_{j=1}^N \|C_j\| + \tau )^{-1}$, for all $i \in \mc I$,

\smallskip
\item $(\gamma^k)_{k \in \bb N}$ such that $\gamma^k\in [0, \nu^{-1}]$ for all $k \in \bb N$ and $\sum_{k=0}^\infty\gamma^k(1-\nu \gamma^k)=\infty$.
{\hfill $\square$}
\end{enumerate}
\end{assumption}

\smallskip
Note that the design choice $\gamma^k = 1 $, for all $k \in \bb N$, always satisfies Assumption \ref{ass:SS-choice1} (iii).


\begin{figure}
\begin{minipage}{\columnwidth}
\hrule
\smallskip
\textsc{Algorithm 2}. \blue{Distributed  (Full-decision Information)}
\smallskip
\hrule 
\smallskip
\text{Initialization}:
For all $i \in \mc I$: set $x_i^0 \in \Omega_i$, $ \lambda_i^0 \in \R^m_{\geq 0}$; set $\alpha_i$, $\beta_i$ and $(  \gamma^k)_{k \in \bb N}$ as in Assumption \ref{ass:SS-choice1}.
\\[.5em]
Iterate until convergence:\\[.2em]
\hspace*{1em} For all $i \in \mc I$
\begin{align*}
&\left|
\begin{array}{l}
 \text{Local projected pseudo-gradient update}: \\[.2em]
\left|
\begin{array}{l}
\tilde x_i^{k} = \proj_{\Omega_i}(x_i^k - \alpha_i( F_i(x_i^k,\bar x^k) + C_i^\top \bar \lambda^k)),\\[.3em]
d_i^{k} = 2 C_i \tilde x_i^{k} - C_i x_i^k - c_i,\\[.3em]
\tilde \lambda_i^{k} = \proj_{\R^m_{\geq 0}}\big( \lambda^k_i + \beta_i ( \bar d^{k} -\lambda^k_i+\bar{\lambda}^k)\big),
\end{array}
\right.\\[2em]
 \text{Local Krasnosel'skii--Mann process:}\\[.2em]
\left|
\begin{array}{l}
x_i^{k+1} = x_i^k + \gamma^k (\tilde x_i^{k} - x_i^k),\\[.3em]
\lambda_i^{k+1} = \lambda_i^k + \gamma^k (\tilde \lambda_i^{k} - \lambda_i^k),
\end{array}
\right.
\end{array}
\right.
\end{align*}
\smallskip
\hrule
\end{minipage}
\end{figure}

\bigskip
\begin{theorem} \label{th:conv-partDistr}
Let Assumptions \ref{ass:ConvCompSet}, \ref{ass:COCO} hold.
If the step-sizes $\{\alpha_i, \beta_i\}_{i \in \mc I}$ and $(\gamma^k)_{k \in \bb N}$ are set as in Assumption \ref{ass:SS-choice1}, then the sequence $( \col(\bs x^k, \bs \lambda^k))_{k \in \bb N}$ generated by Algorithm 2 converges to some $\col(\bs x^*, \bs \lambda^*) \in \zer(T)$, where $\bs x^*$ is a v-GNE of the game in \eqref{eq:Game}.
{\hfill $\square$ }
\end{theorem}
\begin{proof}
See Appendix \ref{pr:proofTC}.
\end{proof}

\smallskip
\begin{remark}[Algorithm 2 as a fixed-point iteration] \label{rem:FPI}
Our convergence analysis is based on the same operator theoretic framework in \cite{yi2019operator}-\cite{belgioioso2018projected}. Specifically, we recast the dynamics generated by Algorithm 2 as the fixed-point iteration
\begin{align} \label{eq:R}
\bs \omega^{k+1}=
\bs \omega^k + \gamma^k( R(\bs \omega^k) - \bs \omega^k), \quad (k \in \bb N)
\end{align}
where $\bs \omega^k = \col(\bs x^k, \bs \lambda^k)$ is the stacked vector of the iterates and $R$ is the so-called pFB operator, defined as
\begin{align} \label{eq:Rmap}
R:= (\Id+\Phi^{-1} T_2)^{-1}\circ(\Id-\Phi^{-1} T_1),
\end{align}
where $T_1$, $T_2$ in \eqref{eq:T1}-\eqref{eq:T2} characterize the splitting of $T$, and $\Phi$ is the so-called preconditioning matrix, here chosen as
\begin{align} \label{eq:SandPhi}
\Phi := \begin{bmatrix}
{\alpha}_\text{d}^{-1} &  -\frac{1}{N} C_{\text{f}}^\top \\
 -\frac{1}{N} C_{\text{f}} &  \beta_\text{d}^{-1}
\end{bmatrix},
\end{align}
$ \alpha_\text{d} := \diag(\alpha_1,\ldots,\alpha_N) \otimes I_n$, $ \beta_\text{d} :=\diag(\beta_1,\ldots,\beta_N)\otimes I_n$.
Then, we show that, if the step sizes in the main diagonal of $\Phi$ are set according to Assumption \ref{ass:SS-choice1}, the mapping $R$ is averaged with respect to the $\Phi$-induced norm, i.e., $\|\cdot \|_{\Phi}$. Hence, the fixed-point iteration \eqref{eq:R} converges to some $\bs \omega^*:= \col(\bs x^*, \bs \lambda^*) \in \fix(R) = \zer(T)$, where $\bs x^*$ is a v-GNE. See Appendix \ref{pr:proofTC} for a complete convergence analysis.
{\hfill $\square$}
\end{remark}

\smallskip
To conclude this section, we note that the projected-pseudo-gradient updates in Algorithm 2 can be cast compactly as
\begin{align} 
&\tilde{\bs x}^{k}= \proj_{\bs \Omega}\big( \bs x^k -  \alpha_\text{d}( \bs F(\bs x^k, \bar{\bs x}^k) + C_\text{d} ^\top \bar{\bs \lambda}^k )\big),
\label{eq:CFalg2_x}
\\
& \tilde{\bs \lambda}^{k} = \proj_{\R^{mN}_{\geq 0}} \big( \bs \lambda^k +  \beta_\text{d} ( \bar{\bs d}^k 
 - \bs \lambda^k + \bar{\bs \lambda}^k) \big),
\label{eq:CFalg2_lam}
\end{align}
where
\begin{align*}
&\bar{\bs x}^k = \1 \otimes \bar x^k, \quad 
\bar{\bs \lambda}^k = \1 \otimes \bar \lambda^k,
\quad \bar{\bs d}^k = \1 \otimes \bar d^k
\end{align*}
and $C_\text{d}:=\textrm{blkdiag}(C_1,\ldots,C_N)$.

\smallskip
\blue{Unlike Algorithm 1, Algorithm 2 does not directly rely on the actions of a central coordinator, namely, dual update and broadcast communication. However, it requires an all-to-all information exchange (or, equivalently, a complete communication graph) at each iteration $k$, since the local updating rule of each agent necessitates the knowledge of:
} 
\begin{enumerate}[(i)]
\item the average strategy $\bar{x}^k$,

\item the average dual variable $\bar \lambda^k $,

\item the average quantity
$\bar d^{k} $.
\end{enumerate}

\subsection{A fully-distributed algorithm via dynamic tracking}
To implement Algorithm 2 fully-distributively under the more realistic communication assumptions in Section \ref{sec:com_network}, we approximate its updates by endowing each agent $i$ with some surrogate variables (or estimates), i.e., $\sigma_i$, $y_i$ and $z_i$, that \textit{dynamically track} the averages $\bar{x}^k$, $ \bar d^k$ and $\bar \lambda^k$, respectively.
Then, to mitigate the errors due to the inexactness of the surrogate variables, we relax the projected-pseudo-gradient iterations by means of a Krasnosel'skii--Mann (KM) process \cite[eq.(5.12)]{bauschke2017convex}, whose step-sizes are set according to the following design choice.

\smallskip
\begin{assumption} \label{ass:van-ss} The sequence $( \gamma^k )_{k \in \bb N}$ satisfies the following conditions:
\begin{enumerate}[(i)]
\item (non-increasing) $0 \le \gamma^{k+1} \le \gamma^k \le 1$, for all $k \ge 0$;
\item (non-summable) $\sum_{k=0}^\infty \gamma^k = \infty$;
\item (square-summable) $\sum_{k=0}^\infty {(\gamma^k)}^2 < \infty$.
{\hfill $\square$}
\end{enumerate}
\end{assumption}

\smallskip
For example, Assumption \ref{ass:van-ss} is satisfied for step sizes of the form $\gamma^k = (k+1)^{-b}$ where $\frac{1}{2} < b \le 1$.

\smallskip
The proposed algorithm relies on agents constructing an estimate of the averages by mixing information drawn from local neighbors and making a subsequent relaxed projected-pseudo-gradient step, as in Algorithm 2.
To build the estimates $\sigma_i$, $y_i$, $z_i$, at every iteration $k$, agent $i$ receives $\sigma_j^k$'s, $y_j^k$'s, $z_j^k$'s from its neighbors, $j \in \mc N_i(k)$, and aligns its intermediate estimates according to the following rules:
\begin{align*}
&\hat{\sigma}_i^k := \sum_{j =1}^N w_{i,j}(k)\sigma_j^k, \quad \hat{y}_i^k := \sum_{j =1}^N w_{i,j}(k)  y_j^k\\
&\hat{z}_i^k := \sum_{j =1}^N w_{i,j}(k) z_j^k
\end{align*}
Then, on the basis of $\hat \sigma^k_i$, $\hat y^k_i$ and $\hat z^k_i$, agent $i$ updates its strategy $x_i^{k+1}$, its dual variable $\lambda_i^{k+1}$ and the new estimates $\sigma_i^{k+1},y_i^{k+1},z_i^{k+
1}$ as formalized in Algorithm 3.
\begin{figure}
\medskip \noindent
\begin{minipage}{\columnwidth}
\hrule
\smallskip
\textsc{Algorithm 3}. \blue{Distributed (Partial-decision Information)}
\smallskip
\hrule 
\smallskip
\text{Initialization}:
For all $i \in \mc I$: set $x_i^{-1}, x_i^0,\tilde{x}_i^{-1} \in \Omega_i$, $ \lambda_i^0 \in \R^m_{\geq 0}$, $\sigma^0_i = x^0_i$, $z_i^0 = \lambda_i^0$, $y_i^0 = 2C_i \tilde{x}_i^{-1}- C_i x_i^{-1} - c_i$; $\alpha_i$, $\beta_i$ as in Assumption \ref{ass:SS-choice1} and $(  \gamma^k)_{k \in \bb N}$ as in Assumption \ref{ass:van-ss}.
\\[.5em]
Iterate until convergence:\\[.2em]
\hspace*{.5em} For all $i \in \mc I$
\begin{align*}
\quad &\left|
\begin{array}{l}
\text{Communication and distributed averaging:}\\
  \left|
  \begin{array}{l}
  \hat{\sigma}_i^k = \sum_{j =1}^N w_{i,j}(k) \sigma_j^k,\\
   \hat y^k_i = \sum_{j =1}^N w_{i,j}(k)  y_j^k, \\
  \hat z^k_i = \sum_{j =1}^N w_{i,j}(k)  z_j^k,
  \end{array}
  \right. \\[2em]
 \text{Local strategy update and dynamic tracking of } \bar d^k:\\[.2em]
\left|
\begin{array}{l}
\tilde x_i^{k} = \proj_{\Omega_i}(x_i^k - \alpha_i( F_i(x_i^k,\hat{\sigma}^k_i) + C_i^\top \hat{z}^k_i)),\\[.3em]
y_i^{k+1} = \hat y^k_i + C_i (2 \tilde x_i^{k} -  x_i^k)- C_i ( 2 \tilde x_i^{k-1} -  x_i^{k-1}),\\[.3em]
\tilde \lambda_i^{k} = \proj_{\R^m_{\geq 0}}\big( \lambda^k_i + \beta_i ( y_i^{k+1} -\lambda^k_i+\hat{z}^k_i)\big),
\end{array}
\right.\\[2em]
 \text{Local Krasnosel'skii--Mann process:}\\[.2em]
\left|
\begin{array}{l}
x_i^{k+1} = x_i^k + \gamma^k (\tilde x_i^{k} - x_i^k),\\[.3em]
\lambda_i^{k+1} = \lambda_i^k + \gamma^k (\tilde \lambda_i^{k} - \lambda_i^k),
\end{array}
\right.\\[1.5em]
  \text{Local dynamic tracking of } \bar x^{k+1} \text{ and } \bar \lambda^{k+1}: \\[.2em]
  \left|
\begin{array}{l}
 \sigma_i^{k+1} =  \hat{\sigma}_i^k + x_i^{k+1} - x_i^k,\\
 z_i^{k+1} = \hat{z}^k_i + \lambda_i^{k+1}- \lambda_i^k.
 \end{array}
 \right.
\end{array}
\right.
\end{align*}
\smallskip
\hrule
\end{minipage}
\end{figure}

\smallskip
Note that the projected-pseudo-gradient updates in Algorithm 3 can be recast in a compact form as
\begin{align}
\label{eq:ALG3compactX_alg3}
&\tilde{\bs x}^{k} \displaystyle
= \proj_{\bs \Omega}\big( \bs x^k -  \alpha_{\text d} (  \bs F(\bs x^k,\hat{\bs \sigma}^{k} ) + C_{\text d}^\top \hat{\bs z}^{k} \big) ,\\
\textstyle
\label{eq:ALG3compact3_lam}
&\tilde{\bs \lambda}^{k} =\displaystyle
 \proj_{\R^{mN}_{\geq 0}} \big( \bs \lambda^k +  \beta_{\text d} ( \bs y^{k+1}
 -\bs \lambda^k + \hat{\bs z}^{k}) \big)
\end{align}
where
\begin{align*}
\textstyle \nonumber
& \hat{\bs \sigma}^{k} = W_n(k) \bs \sigma^k, \quad \hat{\bs z}^{k} = W_m(k)\bs z^k, \quad  \hat{\bs y}^{k} = W_m(k) \bs y^k, \\
&\bs y^{k+1} =\textstyle
\hat{\bs y}^{k}+ C_{\text d}(  2\tilde{\bs x}^k - \bs x^k)
- C_{\text d}(  2 \tilde{\bs x}^{k-1} - \bs x^{k-1}).
\end{align*}
and $W_\ell(k) := W(k) \otimes I_\ell$ for some $\ell \in \bb N$.

\section{Convergence analysis} \label{sec:CA}
\label{sec:ConvA}
To prove the convergence of Algorithm 3, we rely on the framework of the inexact Krasnosel'skii--Mann fixed-point iteration \cite[Alg. 5.4]{combettes2001quasi}.
Informally speaking, our goal is to show that the error deriving from the inexactness of the estimates $\sigma_i$'s, $y_i$'s and $z_i$'s \blue{vanishes to zero fast enough}, in which case, also $(\bs x^k)_{k \in \bb N}$ generated by Algorithm 3 globally converges to a v-GNE.
Technically, we aim at exploiting \cite[Th. 5.5]{combettes2001quasi}, which establishes convergence of an inexact version of the KM iteration in \eqref{eq:R}, i.e., 
\begin{equation} \label{eq:KMe}
\bs \omega^{k+1} = \bs  \omega^k + \gamma^k( R( \bs \omega^k) + e^k - \bs  \omega^k ), \quad \forall k \ge 0,
\end{equation}
when $R$ is nonexpansive and the step-size and error sequences, $( \gamma^k)_{k \in \bb N}$ and $( e^k )_{k \in \bb N}$, respectively, satisfy

\smallskip
\begin{enumerate}[(C.1)]
\item $\sum_{k =0}^\infty \gamma^k(1-\gamma^k) = \infty$,

\smallskip
\item $\sum_{k =0}^\infty \gamma^k \norm{e^k} < \infty$.
\end{enumerate}

\smallskip
Note that Algorithm 3 can be written as the KM with errors in \eqref{eq:KMe} where $\bs \omega^k = \col(\bs x^k, \bs \lambda^k)$ and the error at stage $k$ is
\begin{equation} \label{eq:errDef}
e^k= \col(\tilde{\bs x}^k, \tilde{\bs \lambda}^k) - \col(\tilde{\bs x}^k_{\text{A2}}, \tilde{\bs \lambda}^k_{\text{A2}}),
\end{equation}
where $\tilde{\bs x}^k_{\text{A2}}$ and $\tilde{\bs \lambda}^k_{\text{A2}}$ denote the iterates generated by Algorithm 2 (defined in \eqref{eq:CFalg2_x} and \eqref{eq:CFalg2_lam}, respectively). In other words, $e^k$ represents the distance between the iterates in the ideal case of full-decision information (i.e., where the agents have an exact knowledge of the averages $\bar{x}^k$, $\bar{d}^k$ and $\bar{\lambda}^k$) and the iterates of Algorithm 3, in which the averages are replaced by the estimates $\hat {\sigma}_i^k$, $\hat {y}_i^k $ and $\hat{z}_i^k$, \blue{built on-line by mixing information drawn from local neighboring agents only.}

\smallskip
The main technical challenge to invoke \cite[Th. 5.5]{combettes2001quasi} and, in turn, prove the convergence of Algorithm 3 is to find a step-size sequence $(\gamma^k)_{k \in \bb N}$, that complies with (C.1), such that the relaxed error sequence $( \gamma^k \|e^k \|)_{k \in \bb N}$ satisfies (C.2).
We immediately note that if $(\gamma^k)_{k \in \bb N}$ is chosen as in Assumptions \ref{ass:van-ss}, then it already satisfies (C.1). In the following subsection, we show that (C.2) is also satisfied.

\subsection{Analysis of the relaxed error sequence}
In the next lemma, we recall a fundamental invariance property of dynamic tracking, namely, at each stage $k$, the averages among the estimates $\sigma^k_i$'s, $y^k_i$'s, and $z^k_i$'s are equivalent to the correspondent averages we aim to track. 

\smallskip
\begin{lemma} \label{lem:ESTeRAV} Let Assumption \ref{ass:DSMM} hold true and set the initial conditions $\sigma^0_i, y_i^0, z_i^0$ as in Algorithm 3, for all $i \in \mc I$. Then, the following equations hold for all $k \geq 0$:
\begin{enumerate}[(i)]
\item $\bar \sigma^k  =
\frac{1}{N} \sum_{i=1}^N \sigma_i^k = \bar{x}^k$;
\item $\bar y^k = \frac{1}{N} \sum_{i=1}^N y_i^k
 \textstyle = \bar{d}^k$;
\item $\bar z^k =
 \frac{1}{N} \sum_{i=1}^N z_i^k
 \textstyle = \bar \lambda^k$.
{\hfill $\square$}
\end{enumerate}
\end{lemma}
\begin{proof}
See Appendix \ref{proof:lem:ESTeRAV}.
\end{proof}

\smallskip
\blue{The following assumption on the dual sequences generated by Algorithm 3 is instrumental for the subsequent lemma.
\begin{assumption} \label{ass:BDV}
The sequence $(\bs \lambda^k)_{k \in \bb N}$ generated by Algorithm 3 is bounded, i.e., there exists $B_D > 0$ such that $\| \bs \lambda^k \| \leq B_D$, for all $k\geq 0$.
{\hfill $\square$} 
\end{assumption}
}

\smallskip
\blue{
For example, in the context of distributed constrained optimization, 
Assumption \ref{ass:BDV} can be enforced by changing the local dual updates by projecting onto a local bounded set $D_\text{i}$ that contains the optimal dual set $D^*$ \cite{nedic2009subgradient}, \cite{zhu2012distributed}. See Section \ref{subsec:BDV} for a discussion on how to locally build such supersets.
}

\smallskip
The next lemma provides upper bounds for the estimation errors at each stage $k$ of Algorithm 3.


\smallskip
\begin{lemma} \label{lem:SurrToReal}
Let Assumptions \ref{ass:ConvCompSet}-\ref{ass:DSMM}, \blue{\ref{ass:BDV}} hold true. Then, there exist some positive constants $B_\Omega$, $B_D$, $B_Y$, $\delta_1$ and $\delta_2$ and a vanishing scalar sequence $(\phi^k)_{k \in \bb N}$ defined as
\begin{align} \label{eq:auxSeq}
\textstyle
\phi^k = \delta_1 \rho^{k-1} + \delta_2 \sum_{\ell=1}^k \rho^{k-\ell} \gamma^{\ell - 1},
\end{align}
with $\rho$ as in \eqref{eq:rho} and $(\gamma^k)_{k \in \bb N}$ as in Assumption \ref{ass:van-ss}, such that the following upper bounds hold for all $k \in \bb N$:
\begin{enumerate}[(i)]
\item $\| \hat{ \bs \sigma}^k -  \1  \otimes \bar x^k  \|  \textstyle  \leq \theta B_{\Omega} \rho^k  + \theta B_{\Omega} \sum_{s=1}^k  \rho^{k-s} \gamma^{s-1}$;

\item $\| \hat{ \bs z}^k  - \1  \otimes \bar \lambda^k  \| \textstyle  \leq \theta B_{D}  \rho^k  + \theta B_{D} \sum_{s=1}^k  \rho^{k-s} \gamma^{s-1}$;

\item $\|  \bs y^{k+1}  - \1  \otimes \bar d^k    \| \textstyle  \leq \theta B_Y \rho^k  +  \sum_{s=1}^k  \rho^{k-s} \phi^{s-1} + \phi^k$.
\end{enumerate}
\end{lemma}
\begin{proof}
See Appendix \ref{proof:SurrToReal}.
\end{proof}

\smallskip
By exploiting the upper bounds in Lemma \ref{lem:SurrToReal} and a result on the convergence of scalar sequences, which is recalled next, we can show that the estimates asymptotically converge to their correspondent aggregate true values.
\smallskip
\begin{lemma}[{\cite[Lemma 3.1]{ram2010distributed}}]
\label{lem:ram}
Let $(\delta^k)_{k \in \bb N}$ be a sequence.
\begin{enumerate}[(a)]
\item If $\lim_{k \rightarrow \infty} \delta^k = \delta$ and $0 <\tau < 1$, then $\lim_{k \rightarrow \infty} \sum_{\ell = 0}^k \tau^{k-\ell}\delta^\ell = \delta/(1-\tau)$.

\smallskip
\item If $\delta^k \geq 0$ for all $k$, $\sum_{k = 0}^\infty \delta^k < \infty$ and $0 <\tau < 1$, then
$\sum_{k=0}^\infty \sum_{\ell = 0}^k \tau^{k-\ell}\delta^\ell < \infty$.
{\hfill $\square$}
\end{enumerate}
\end{lemma}

\smallskip
\begin{proposition}
Let Assumptions \ref{ass:ConvCompSet}-\ref{ass:DSMM} hold true. Then, the following statements hold:
\begin{enumerate}[(i)]
\item $\lim_{k \rightarrow \infty} \|\hat{ \bs \sigma}^k - \1  \otimes \bar x^k  \| = 0 $;

\item $\lim_{k \rightarrow \infty} \| \hat{\bs z}^k - \1  \otimes \bar \lambda^k  \| = 0 $;

\item $\lim_{k \rightarrow \infty} \| \bs y^{k+1} - \1  \otimes \bar d^k  \| = 0 $.
{\hfill $\square$}
\end{enumerate}
\end{proposition}
\begin{proof}
(i) From the upper bound in Lemma \ref{lem:SurrToReal} (i), we have
\begin{multline*}
\limsup_{k \rightarrow \infty} \|(W(k)  \otimes I_n) \bs \sigma^k - \1  \otimes \bar x^k \| \\
\leq  \limsup_{k \rightarrow \infty} \left( \theta B_{\Omega} \rho^k  + \theta B_{\Omega} \sum_{s=1}^k  \rho^{k-s} \gamma^{s-1} \right) \leq 0,
\end{multline*}
where $\lim_{k \rightarrow \infty} \rho^k = 0$, since $0<\rho<1$ by Lemma \ref{lem:Tsi}, and $\lim_{k \rightarrow \infty} \sum_{s=1}^k  \rho^{k-s} \gamma^{s-1}=0$ by Lemma \ref{lem:ram} (a), since $0<\rho<1$ and $\lim_{k \rightarrow \infty} \gamma^k = 0$ by Assumption \ref{ass:van-ss}. Hence, $\lim_{k \rightarrow \infty} \|\hat{\bs \sigma}^k - \1  \otimes \bar x^k  \| = 0 $. The proofs of (ii) and (iii) are analogous.
\end{proof}

\medskip
Next, we derive an upper bound for the error $e^k$ in \eqref{eq:errDef} that directly depends on the estimation errors in Lemma \ref{lem:SurrToReal}.

\smallskip
\begin{lemma} \label{lem:BoundErr}
Let Assumptions \ref{ass:ConvCompSet}-\ref{ass:DSMM}, \blue{\ref{ass:BDV}} hold true. 
Then, the following bound holds for all $k \in \bb N$:
\begin{align*}  \nonumber
\| e^k \| &\leq L_{\tilde{F}} \|\alpha_{\text d} \|  \|\hat{ \bs \sigma}^k - \1 \otimes \bar{x}^k \| +  \|\beta_{\text d} \| \| \bs y^{k+1}- \1 \otimes \bar{d}^k\| \\
 \nonumber
&\quad +\left(\|\alpha_{\text d} \| \|C_d \| + \|\beta_{\text d} \|\right)\|\hat{ \bs z}^k- \1 \otimes \bar{\lambda}^k \|.
\end{align*}
\end{lemma}
\begin{proof}
See Appendix \ref{proof:lem_BoundErr}.
\end{proof}

\smallskip
Finally, by combining the upper bounds in Lemma \ref{lem:SurrToReal} and \ref{lem:BoundErr} and exploiting a result on the convergence of scalar sequences, i.e., Lemma \ref{lem:ram} (b), we show that condition (C.2) holds, namely, the relaxed error sequence $(\gamma^k \|e^k\|)_{k \in \bb N}$ is summable. 

\smallskip
\begin{lemma} \label{lem:summ-Err}
Let Assumptions \ref{ass:ConvCompSet}-\ref{ass:DSMM}, \blue{\ref{ass:BDV}} hold true. The sequence $(\gamma^k \|e^k\| )_{k \in \bb N}$, with $e^k$ as in \eqref{eq:errDef}, is summable, i.e.,
\begin{equation*} \label{eq:summ:Err} 
\sum_{k=0}^\infty \gamma^k \| e^k \| < \infty.
\end{equation*}
\end{lemma}
\begin{proof}
See Appendix \ref{proof:lem:summ-Err}.
\end{proof}

\smallskip
Now, we can prove the convergence of Algorithm 3.

\smallskip
\begin{theorem}\label{th:TC-KM}
Let Assumptions \ref{ass:ConvCompSet}-\ref{ass:LC-TF}, \blue{\ref{ass:BDV}} hold true, the step sizes $\{\alpha_i, \beta_i\}_{i \in \mc I}$ be set as in Assumption \ref{ass:SS-choice1}, and $(\gamma^k)_{k \in \bb N}$ as in Assumption \ref{ass:van-ss}. Then, the sequence $(\col(\bs x^k, \bs \lambda^k ))_{k \in \bb N}$ generated by Algorithm 3 globally converges to
some $\col(\bs x^*, \bs \lambda^*) \in \zer(T)$, where $\bs x^*$ is a v-GNE of the game in \eqref{eq:Game}.
{\hfill $\square$}
\end{theorem}

\smallskip
\begin{proof}
For all $k \in \bb N$, the iterations of Algorithm 3 can be cast as the Krasnosel'skii--Mann process with errors
$\bs \omega^{k+1} = \bs \omega^k + \gamma^k( R(\bs \omega^k) + e^k - \bs \omega^k )$, where $\bs \omega^k = \col(\bs x^k, \bs \lambda^k)$, $R$ as in \eqref{eq:Rmap} and $e^k$ as in \eqref{eq:errDef}.
By \cite[Th. 5.5]{combettes2001quasi}, the sequence $( \bs \omega^k )_{k \in \bb N}$ converges to some $\bs \omega^* \in \fix(R)$, since $R$ is averaged, thus nonexpansive, by Lemma \ref{lem:Raveraged}, and (C.1)$-$(C.2) hold, by Assumption \ref{ass:van-ss} and Lemma \ref{lem:summ-Err}, respectively. To conclude, we note that $\bs \omega^* \in \fix(R)= \zer(\Phi^{-1} T_1+\Phi^{-1}T_2)$, by \cite[Prop. 25.1 (iv)]{bauschke2017convex}, and that $\zer(\Phi^{-1}T_1+\Phi^{-1}T_2)= \zer(T) \neq \varnothing$, with $T$ as in \eqref{eq:T}, since $\Phi\succ 0$, by Lemma \ref{lem:PMa}, and $T_1+T_2=T$. Since $\bs \omega^* \in \zer(T) $, then $\bs x^*$ is a v-GNE of the game in \eqref{eq:Game}, by Proposition \ref{pr:EOC} (ii).
\end{proof}

\section{Numerical simulations}
\label{sec:NS}
In this section, we study the performance of the proposed algorithm on a class of network Nash--Cournot games with market capacity constraints. Such games represent an instance of generalized aggregative Nash games. In Section \ref{subsec:NCG}, we describe the player cost functions and strategy sets and verify that the necessary assumptions are satisfied. In Section \ref{subsec:SR}, we compare the performance of our algorithm against a standard semi-decentralized method (Algorithm 1).

\subsection{Generalized network Nash--Cournot game}
\label{subsec:NCG}
We extend the network Nash--Cournot game model proposed in \cite[\S IV]{koshal:nedic:shanbhag:16} with additional market capacity constraints. 
Specifically, consider $N$ firms that compete over $m$ markets. Let firm $i$'s production and sales at location $l$ be denoted by $g_{i,l}$ and $s_{i,l}$, respectively, while its
cost of production at location $l$ is denoted by $f_{i,l}(g_{i,l})$ and defined as follows:
\begin{equation} \label{eq:ProdCost}
f_{i,l}(g_{i,l}) = a_{i,l}g_{i,l}^2+ g_{i,l}b_{i,l},
\end{equation}
where $a_{i,l}$ and $b_{i,l}$ are scaling parameters for agent $i$.

The goods sold by firm $i$ at location $l$ fetch a revenue $p(\bar s_l)s_{i,l}$, where $p(\bar s_l)$ denote the sales price at location $l$ and $\bar s_l = \sum_{i=1}^N s_{i,l}$ represents the aggregate sales at location $l$. The market price is set according to an inverse
demand function which depends on the aggregate of the network, i.e.,
\begin{equation*}
p_l(\bar s_l) = d_l - \bar s_l,
\end{equation*}
where $d_l$ is the overall demand for location $l$.
Each firm $i$ has a production limitation at location $l$, described by $u_{i,l}$. Moreover, the overall production in each market $l $ must meet the correspondent demand $d_l$ and do not exceed a maximum capacity $r_l$. Hence, the coupling constraints $d_l \leq \sum_{i=1}^N g_{i,l} \leq r_l$, for all $l = 1,2,\ldots, m$, have to be satisfied.

Overall, each firm $i$, given the strategies of the other firms, aims at solving the following optimization problem:
\begin{align*} \textstyle 
\left\{
\begin{array}{c l}
\underset{ \{ g_{i,l}, s_{i,l} \}_{l=1}^m }{\argmin} &
\sum_{l=1}^m (f_{i,l}(g_{i,l}) - p_l(\bar s_l)s_{i,l} )
\\
\text{s.t.} & 
\sum_{l=1}^m g_{i,l} \geq \sum_{l=1}^n s_{i,l},
 \\[0.2em]
& g_{i,j}, s_{i,j} \geq 0, \; g_{i,l} \leq u_{i,l}, \quad l = 1, \ldots,m,\\[.2em]
& d_l \leq \sum_{i=1}^N g_{i,l} \leq r_l,   \qquad l = 1, \ldots, m.
\end{array}
\right.
\end{align*}

\smallskip

Effectively, the payoff function of firm $i$ is parametrized by nodal aggregate sales and its constraints depend on the other firms' strategies, thus leading to a generalized aggregative game.
In this example, we assume that the firms communicate over a dynamic network to cope with the lack of aggregate information, which is necessary to compute their optimal production and sale strategies.

\smallskip
Next, we show that the proposed network Nash--Cournot game does satisfy our technical setup. Let $x_i = \col(g_{i,1}, \ldots, g_{i,m},s_{i,1},\ldots,s_{i,m}) \in \R^{2m}$ denote the strategy vector of agent $i$ and $\bs x = \col(x_1, \ldots, x_N)$ denote the collective strategy profile. The cost function of agent $i$ is quadratic, convex in $x_i$, continuously differentiable and can be cast in a compact form as 
\begin{equation}
J_i(x_i, \bar{x}) = x_i^\top A_i x_i + b_i^\top x_i +(\Delta \bar{x})^\top x_i,
\end{equation}
where $A_i := \diag(a_{i,1},\ldots,a_{i,m},0,\ldots,0)$,  $\Delta= \diag(\0,I_m)$ and $b_i := \col(b_{i,1},\ldots,b_{i,m}, -d_{1},\ldots, -d_{n})$. The local feasible set of firm $i$ is non-empty (for an adequate choice of $u_{i,l}$'s ), convex, compact and reads as $\Omega_i := \{x_i \in \R^{2n} \, | \, \sum_{l=1}^n g_{i,l} \geq \sum_{l=1}^n s_{i,l}, \; g_{i,j}, s_{i,j} \geq 0, \; g_{i,l} \leq u_{i,l}, \;l = 1, \ldots, m, \}$. 

The coupling constraints are affine and can be written in compact form as in \eqref{eq:Game}, with $C_i =
\left[
\begin{smallmatrix}
\0 & I_m \\
\0 & -I_m
\end{smallmatrix}
\right]
$ and $c_i = \frac{1}{N} \col(r_1,\ldots,r_m, -d_1, \ldots, -d_m)$, for all $i \in \mc I$. Thus, Assumption \ref{ass:ConvCompSet} is satisfied.

The pseudo gradient mapping $F$ is affine and reads as
\begin{align} \textstyle
F(\bs x) = P \bs x
 + b,
\end{align}
with
\begin{equation} \textstyle \label{eq:matrixP}
P = 2A + \frac{1}{N}I\otimes \Delta + \frac{1}{N}(\1 \1^\top \otimes \Delta),
\end{equation}
$A=\textrm{blkdiag}(A_1,\ldots,A_N)$ and $b = \col(b_i,\ldots,b_N)$.
By a direct inspection of the eigenvalues of $P$, we can show that $F$ is strongly monotone and Lipschitz continuous, when the coefficients $a_{i,j}$'s are positive. Hence, Assumption \ref{ass:COCO} is satisfied. In particular, it follows by \cite[p.79]{facchinei2007finite} that $F$ is $\chi-$cocoercive with $\chi := \| P\|^{-1}$. Moreover, since $F$ is strongly monotone and the sets $\Omega_i$ are compact, it follows by Remark \ref{rem:REU} that there exists a unique v-GNE.
The mapping $\tilde{F}$ is affine and reads as
\begin{align*} \textstyle
\bs F(\bs x,\bs \sigma) = (2A + \frac{1}{N}I\otimes \Delta) \bs x
+ (I \otimes \Delta) \bs \sigma + b.
\end{align*}
Similarly, it can be shown that $\bs F$ is $L_{\bs F}-$Lipschitz continuous with $L_{\bs F} := \max_{ij} \{a_{i,j}, 1 \}$. Thus, Assumption \ref{ass:LC-TF} is satisfied.

\subsection{Simulations studies}
\label{subsec:SR}
In our numerical study we consider a network Nash-Cournot game played by $20$ firms, i.e., $N=20$, over $10$ markets, i.e., $m=10$. All the parameters of the game are drawn from uniform distributions and fixed over the course of the entire simulations. Specifically, for all $i \in \mc I$ and $l \in \{1,\ldots,m \}$, we set the parameters of production cost in \eqref{eq:ProdCost} as $a_{i,l} \in \mc U(2,3)$ and $b_{i,l} \in \mc U(2,12)$, where $\mc U(t,\tau)$ denotes the uniform distribution over an interval $[t,\tau]$ with $t < \tau$. We set the production capacities of firm $i$ as $u_{i,l} \in \mc U(50,100)$ for all $l \in \{1, \ldots,n\ \}$ and for all $i \in \mc I$. Moreover, the demand at market $l$ is set as $d_l \in \mc U(90,100)$, while the market capacity as $r_l \in \mc U (d_l,2 d_l)$ for all $l \in \{1, \ldots,m\ \}$.

\smallskip
At each iteration $k$, the firms communicate according to a randomly generated and connected small world, where each node has 4 neighbors. To create a doubly stochastic mixing matrix $W(k)$, we exploit the Metropolis weighting rules in \eqref{eq:Metropolis}. Thus, Assumptions \ref{ass:ConnGra} and \ref{ass:DSMM} are satisfied.
The agents update their decisions and their estimates as in Algorithm 3. The step-sizes $\{\alpha_i, \beta_i \}_{i \in \mc I}$ are set according to Assumption \ref{ass:SS-choice1}, where the global parameter $\tau$ is set $5 \%$ larger than the theoretical lower bound $\frac{1}{2\delta}$, where $\delta = \min\{1, \|P\| \}$ and $P$ as in \eqref{eq:matrixP}.

\smallskip
In Figure \ref{fig:compAlg}, we show the trajectories of the sequences of normalized residuals $\|\bs x^k - \bs x^* \|/ \| \bs x^0-\bs x^* \|$ for different choices of the step-size sequence $(\gamma^k)_{k\in \bb N}$.
Moreover, we compare the trajectories of Algorithm 3 with those obtained with Algorithm 1 \cite[Alg. 1]{belgioioso2018projected}, which is a semi-decentralized algorithm and works under the assumption of full-decision information, i.e., the firms have access to the real aggregate information at each stage $k$ of the algorithm. As expected, the semi-decentralized algorithm converges faster than the fully-distributed counterpart. Interestingly, we notice that convergence is achieved also in the case of fixed relaxation step in the KM process, e.g. $\gamma^k = 1$ for all $k \geq 0 $, which is not supported by our theoretical analysis.

In Figure \ref{fig:compAlgCD}, we compare the trajectories of the consensus disagreement of the dual variables $\|(\L \otimes I_m) \bs \lambda^k\|$ for two choices of the step-size sequence $(\gamma^k)_{k\in \bb N}$.

\begin{figure}[t]
\includegraphics[width=\columnwidth]{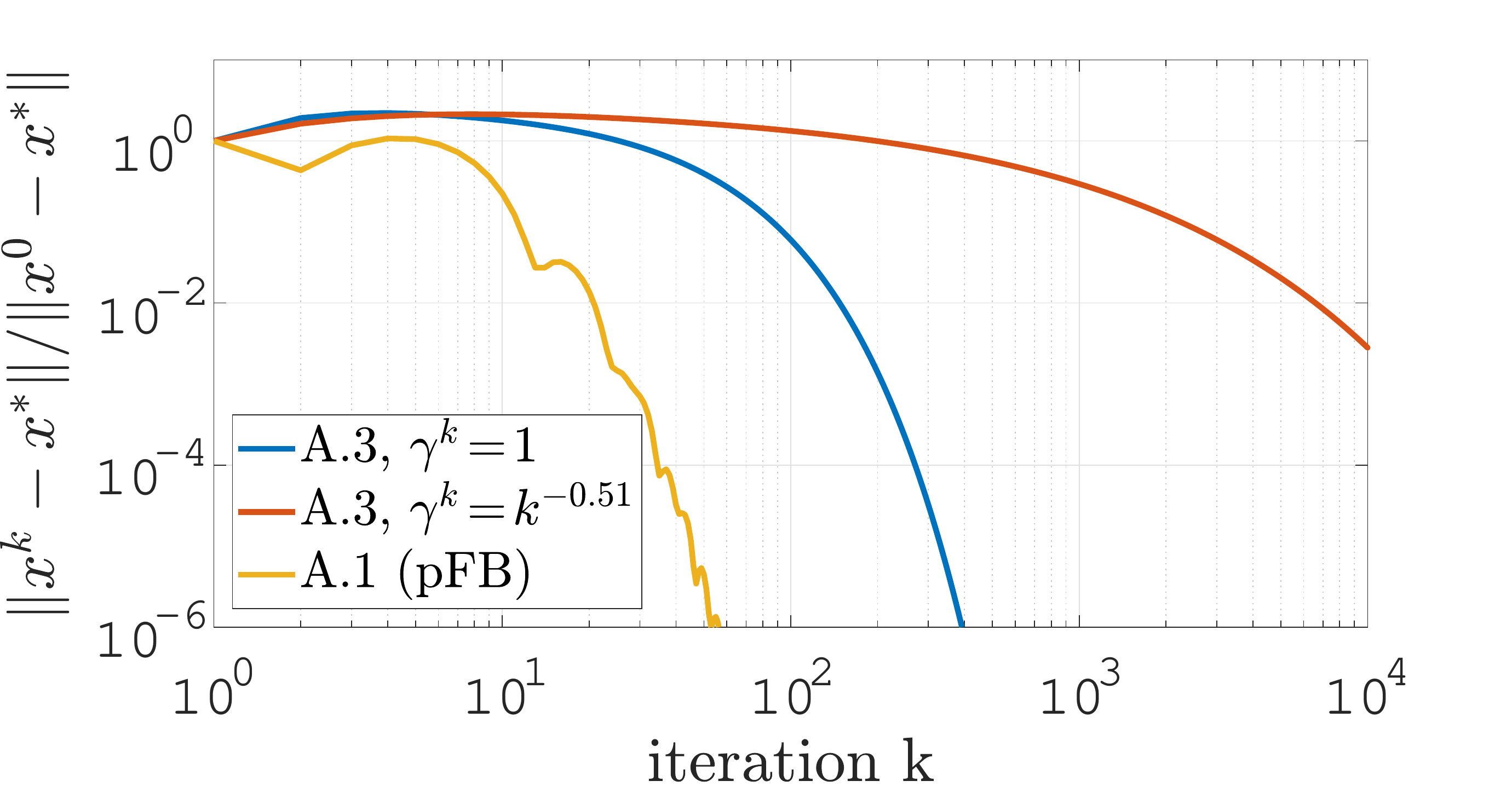}
\caption{The trajectories of the residual $\|\bs x^k - \bs x^* \|/\| \bs x^0 - \bs x^*  \|$ for pFB \cite[Alg. 1]{belgioioso2018projected},
Alg. 3 with $\gamma^k = k^{-0.51}$ and Alg. 3 with $\gamma^k =1$.}
\label{fig:compAlg}
\end{figure}
\begin{figure}[t]
\includegraphics[width=\columnwidth]{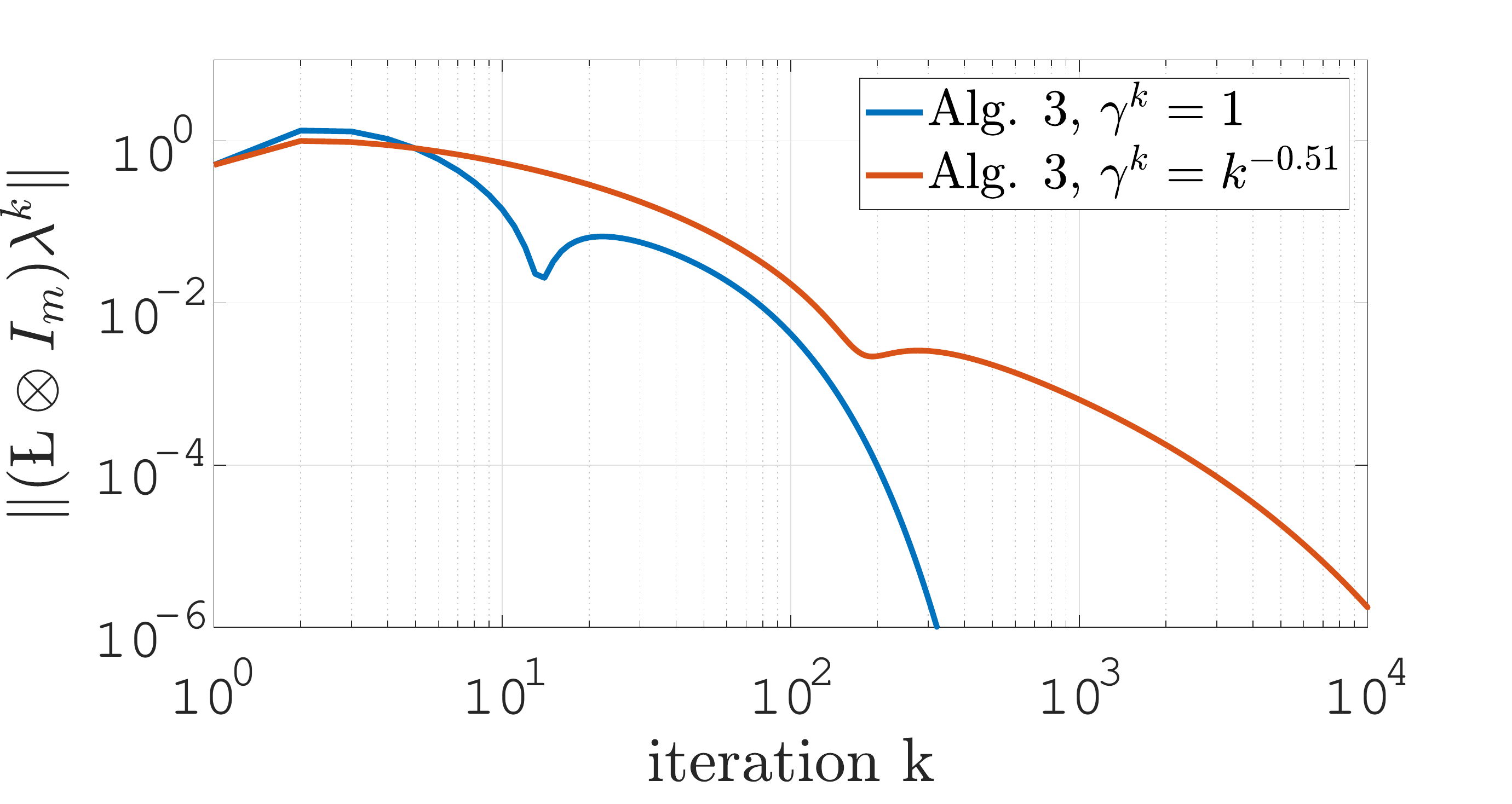}
\caption{The trajectories of the consensus disagreement $\|(\L \otimes I_m) \bs \lambda^k\|$ for Alg. 3 with $\gamma^k = k^{-0.51}$ and Alg. 3 with $\gamma^k =1$.}
\label{fig:compAlgCD}
\end{figure}	
\begin{figure}[h]
\includegraphics[width=\columnwidth]{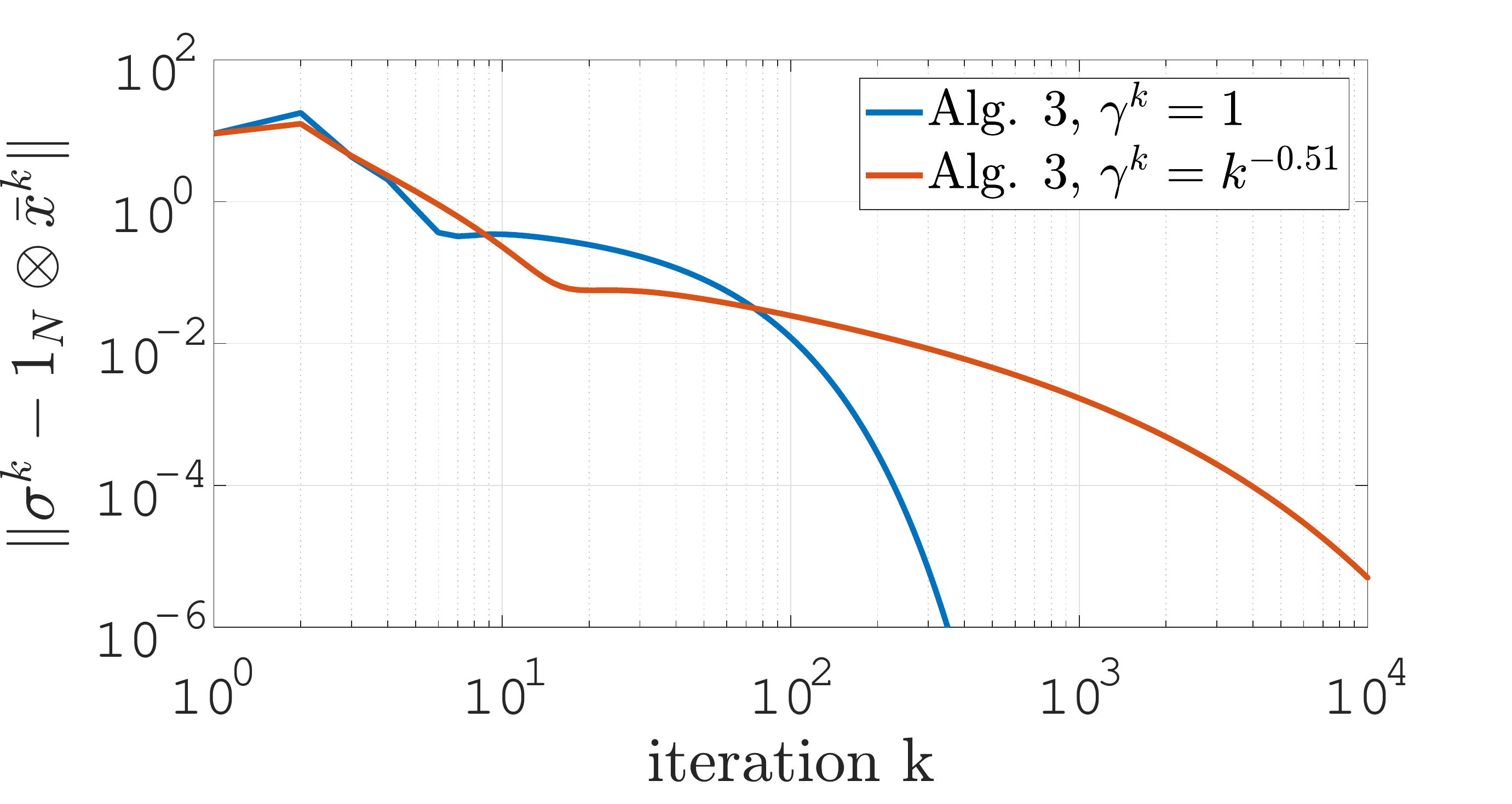}
\caption{The trajectories of the overall tracking error $\|\bs \sigma^k - \1  \otimes \bar x^k\|$ for Alg. 3 with $\gamma^k = k^{-0.51}$ and Alg. 3 with $\gamma^k =1$.}
\label{fig:compAlgTE}
\end{figure}
\begin{figure}[h]
\includegraphics[width=\columnwidth]{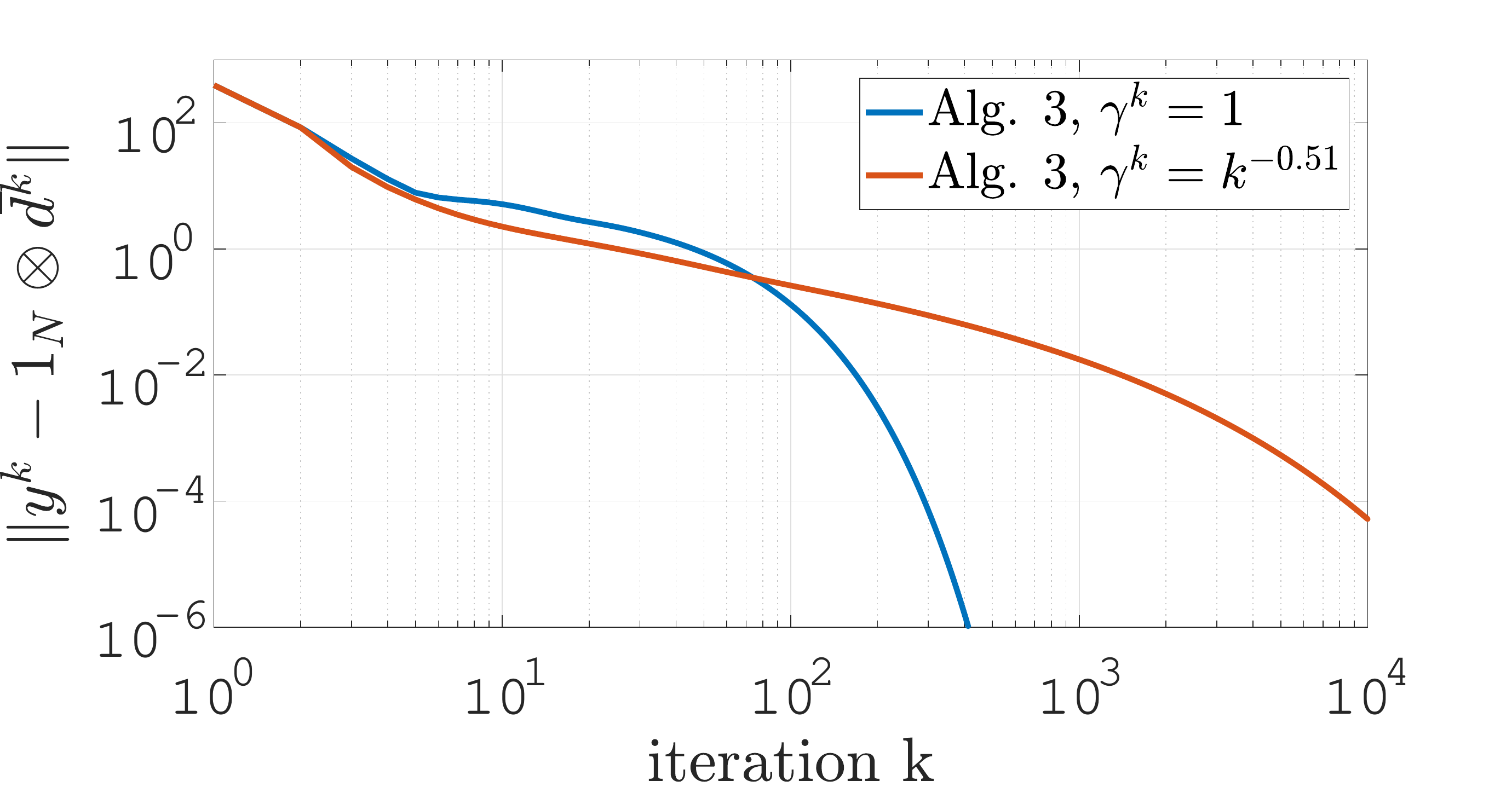}
\caption{The trajectories of the overall tracking error $\|\bs y^k - \1  \otimes \bar d^k\|$ for Alg. 3 with $\gamma^k = k^{-0.51}$ and Alg. 3 with $\gamma^k =1$.}
\label{fig:compAlgTECV}
\end{figure}	

\section{Conclusion}
\label{sec:Concl}
For a general class of aggregative games with linear coupling constraints over time-varying communication networks, we have designed the first single-layer, fully-distributed algorithm to compute a variational generalized Nash equilibrium. Global convergence can be established via monotone-operator-theoretic and fixed-point arguments, integrated with a dynamic tracking methodology. 

The analysis approach in this paper is genuinely novel, hence opens up a number of new research directions. Motivated by the numerical results of Section \ref{sec:NS}, it would be valuable to explore the computational aspects of the proposed method, e.g. how the connectivity of the communication networks influences the convergence speed. 
Whether or not the proposed algorithm converges with fixed step sizes in the Krasnosel'skii-Mann process is currently an open question. Finally, it would be highly valuable to relax the assumption of double-stochasticity of the mixing matrices.

\appendix

\subsection{Proof of Lemma \ref{lem:RSM}}
\label{proof:lemRSM}
(i) $T_2$ is the sum of two terms: $S$ in \eqref{eq:SandPhi} which is a linear, skew symmetric mapping, thus maximally monotone \cite[Ex. 20.30]{bauschke2017convex}; and $\nc_{\bs \Omega} \times \nc_{\R^{mN}_{\geq 0}}$ which is maximally monotone since  is the direct sum of maximally monotone operators \cite[Prop. 20.23]{bauschke2017convex} (i.e, the normal cones of the closed convex sets $\bs \Omega$ and $\R^{mN}_{\geq 0}$). Hence, the maximal monotonicity of $S+\nc_{\bs \Omega} \times \nc_{\R^{mN}_{\geq 0}}=T_2$ follows by \cite[Cor. 24.4 (i)]{bauschke2017convex} since $\dom(S)=\R^{(n+m)N}$.

(ii) 
%
$F$ is $\chi-$cocoercive, by Assumption \ref{ass:COCO}, and $\L_m$ is $1-$cocoercive by \cite[p.79]{facchinei:pang}, since $\L_m$ is a linear, positive semi-definite mapping with $\|\L_m\| = 1$. It follows that the direct sum $T_1(\cdot) = F(\cdot) \times (\L_m \cdot + \frac{1}{N} c_{\text f})$ is $\delta-$cocoercive, for all $\delta $ such that $0 < \delta \leq \min \{ 1,\chi \}$.
Now, we show that $T_1$ is restricted-strictly monotone w.r.t. $\Theta^\parallel = \bs \Omega \times \bs E^\parallel$. \blue{Let us recall that $\bs{E}^{\parallel}$ and $\bs{E}^{\perp}$ are the ($m$-dimensional) consensus and disagreement subspaces, respectively. Moreover, each vector $v \in \R^m$, can be split as $v=v_\parallel+v_\perp$, with $v_\parallel \in \bs{E}^{\parallel}$ and $v_\perp \in \bs{E}^{\perp}$}.
Consider now $\bs \omega = \col(\bs x, \bs \lambda) \not \in \Theta^\parallel$, hence $\bs \lambda = \bs \lambda_\parallel + \bs \lambda_\perp$, with $\bs \lambda_\parallel \in \bs{E}^{\parallel}$ and $\0 \neq \bs \lambda_\perp \in \bs{E}^{\perp}$.
Let $\bs \omega'=\col(\bs x',\bs \lambda') \in \Theta^\parallel$, hence $\bs \lambda' = \bs \lambda'_\parallel \in \bs{E}^{\parallel}$ and $\bs \lambda'_\perp = \0$. The following inequalities show that $T_1$ in \eqref{eq:T2} is restricted-strictly monotone w.r.t. $\Theta^\parallel$:
\begin{align*} \textstyle
&\textstyle (T_2(\bs \omega)-T_2(\bs \omega'))^\top (\bs \omega - \bs \omega') \\
& \textstyle
= (F(\bs x)-F(\bs x'))^\top (\bs x - \bs x') 
+ (\bs \lambda - \bs \lambda')^\top \L_m(\bs \lambda - \bs \lambda')\\
 &\textstyle 
 \ge \chi \|F(\bs x) - F(\bs x') \|^2 +  ( \bs \lambda_\perp)^\top \L_m \bs \lambda_\perp\\
 & \textstyle 
 \geq \text{eig}_2(\L)\|\bs \lambda_\perp \|^2>0,
\end{align*}
where $\L_m = (\L \otimes I_m)$, \blue{with $\L$ projection onto the disagreement subspace} and $\text{eig}_2(\L)=1$ is the second smallest eigenvalue of $\L = I -\frac{1}{N} \1 \1^\top$. The first inequality follows by the cocoercivity of $F$ (Assumption \ref{ass:COCO}) and since $\L_m \bs \lambda_\parallel = \L_m \bs \lambda' =0$, \blue{
namely, the projection onto the disagreement subspace of the consensual terms is zero}. 

(iii): The maximal monotonicity of $T=T_1+T_2$ follows by \cite[Cor. 24.4 (i)]{bauschke2017convex}, since $T_1$ is cocoercive (thus maximally monotone \cite[Example 20.31]{bauschke2017convex}), $T_2$ is maximally monotone and $\dom(T_1)=\R^{(n+m)N}$.
Moreover, since $T_1$ is also restricted-strictly monotone with respect to $\Theta_\parallel$ then $T $ enjoys the same property.
{\hfill $\blacksquare$}

\subsection{Proof of Proposition \ref{pr:EOC}}
\label{prof:prEOC}
(i) By Proposition \ref{pr:UvGNE}, there exists $ \lambda^* \in \R^m_{\geq 0} $ such that $\col(\bs x^*,\lambda^*)\in \zer(U)$, where $\bs x^*$ is a v-GNE. Define $\bs \omega^* = \col(\bs x^*,\bs \lambda^*)$, with $\bs \lambda^*=\1_N \otimes \lambda^*$, then we have $T(\bs \omega^*) \ni \0$. In fact, each component of the first row block of $T(\bs \omega^*)$ reads as $\nc_{\Omega_i}({x}^*_i) +\nabla_{x_i} J_i(x_i^*,\bar{x}^*) +  C_i^\top \lambda^* \ni \0$. While, each component of the second row block of $T(\bs \omega^*)$ reads as $\nc_{\R^{m}_{\geq 0}}( \lambda^*)  - \frac{1}{N}( C \bs x - c) \ni \0$, since $\nc_{\R^{m}_{\geq 0}}( \lambda^*)  - ( C \bs x^* - c) \ni \0$ and $\frac{1}{N} \nc_{\R_{\ge 0}^m} = \nc_{\R_{\ge 0}^m}$. Hence, $\zer(T)\neq \varnothing$.

(ii) From the first part of the proof, we know that there exists $\bs \omega^* \in \Theta^\parallel$ such that $\bs \omega^* \in \zer(T)$.
Now, we show that all the zeros of $T$ lie in $\Theta^\parallel$. By contradiction, let $\bs \omega' \in \zer(T)$ and assume $\bs \omega' \notin \Theta^\parallel$. Then, $\0 \in T(\bs \omega^* )$, $\0 \in T(\bs \omega' )$ and
Lemma \ref{lem:RSM} (iii) yields $0= (\0-\0)^\top(\bs \omega^* -\bs \omega') > 0$, which is impossible. 
Therefore, $\bs \omega' \in \Theta^\parallel$, namely $ \bs \omega'=\col(\bs x', \1 \otimes \lambda')$. Now, by substituting $\bs \omega'$ into $T$ (since $(\L \otimes I_m)(\1 \otimes \lambda')=\0$) we recover that  $\bs \omega' \in \zer(T) \Rightarrow \col(\bs x', \lambda') \in \zer(U)$, which, by Proposition \ref{pr:UvGNE}, holds if and only if $\bs x'$ is a v-GNE.
%
{\hfill $\blacksquare$}

\subsection{Proof of Theorem \ref{th:conv-partDistr}}
\label{pr:proofTC}
To prove convergence of Algorithm 2 we follow the same technical reasoning of the proof in \cite[Alg. 1]{yi2019operator}. 
Specifically, the proof is divided in two parts to show that:
\begin{enumerate}[(1)]
\item
Algorithm 2 corresponds to the fixed-point iteration in \eqref{eq:R}, i.e., $\bs \omega^{k+1}= \bs \omega^k + \gamma^k(
R(\bs \omega^k)-\bs \omega^k)$, where $R:= (\Id+\Phi^{-1} T_2)^{-1}\circ(\Id-\Phi^{-1} T_1)$ is the so-called pFB operator.
\item If the step sizes are set as in Assumption \ref{ass:SS-choice1}, then $R$ is an averaged operator. Hence, \eqref{eq:R} globally converges to some $\bs \omega^*:=\col(\bs x^* \bs \lambda^*) \in \fix(R)$. Since $\fix(R) = \zer(T)$, with $T$ as in \eqref{eq:T}, then $\bs x^*$ is a v-GNE, by Proposition \ref{pr:EOC}.
\end{enumerate}

\smallskip
(1): Let us recast Algorithm 2 in a compact form as
\begin{align}
\textstyle
\label{eq:ALG2compact1}
&\tilde{\bs x}^{k} \displaystyle
= \proj_{\bs \Omega}\big( \bs x^k -  \alpha_{\text{d}}(  \bs F(\bs x^k, \bar{\bs x}^k) +  C^\top_{\text d} \bar {\bs \lambda}^k )\big) ,\\
\textstyle
\label{eq:ALG2compact2}
&\tilde{\bs \lambda}^{k} =\displaystyle
 \proj_{\R^{mN}_{\geq 0}} \big( \bs \lambda^k + \beta_{\text{d}}( \bar{\bs d}^{k}
 -\bs \lambda^k + \bar{\bs \lambda}^k ) \big)\\
 \label{eq:ALG2compact3}
& \bs x^{k+1} = \bs x^k + \gamma^k (\tilde{\bs x}^{k}- \bs  x^k),\\
 \label{eq:ALG2compact4}
& \bs \lambda^{k+1} = \bs \lambda^k + \gamma^k (\tilde{\bs \lambda}^{k}- \bs \lambda^k),
\end{align}
%
Since $\proj_{\bs \Omega} = (\Id+\nc_{\bs \Omega})^{-1}$, $\bs F(\bs x^k, \1 \otimes \bar{ x}^k) = F(\bs x^k)$ and $C^\top_{\text d} \bar {\bs \lambda}^k =C_{\text d}^\top (\1 \otimes \bar \lambda^k) = \frac{1}{N} C_{\text{f}}^\top \bs \lambda^{k}$, it follows from \eqref{eq:ALG2compact1} that $(\Id+\nc_{\bs \Omega})(\tilde{\bs x}^{k}) \ni \bs x^k - \alpha_{\text{d}}( F(\bs x^k)+ \frac{1}{N} C_{\text{f}}^\top \bs \lambda^k )$, which leads to
\begin{multline} \label{eq:Alg2-X}
\textstyle
-F(\bs x^k) \in \nc_{\bs \Omega}(\tilde{\bs x}^{k}) + \frac{1}{N}  C_{\text f}^\top \tilde{\bs \lambda}^{k}  \\
\textstyle
+ \alpha_{\text d}^{-1}(\tilde{\bs x}^{k}-\bs x^k)
- \frac{1}{N} C_{\text{f}}^\top (\tilde{\bs \lambda}^{k} -\bs \lambda^k),
\end{multline}
where we used $ \alpha_{\text d}^{-1} \nc_{\bs \Omega}(\tilde{\bs x}^{k}) = \nc_{\bs \Omega}(\tilde{\bs x}^{k})$. Similarly, since $ \1 \otimes \bar{d}^k= \frac{1}{N}(2  C_{\text{f}} \tilde{\bs x}^{k} -  C_{\text{f}} \bs x^k- c_{\text{f}})$ and $\bs \lambda^k - \bar{\bs \lambda}^k = ((I - \frac{1}{N} \1 \1^\top)\otimes I_m) \bs \lambda^k = (\L \otimes I_m)\bs \lambda^k = \L_m \bs \lambda^k$, it follows from \eqref{eq:ALG2compact2} that $(\Id+\nc_{\R^{mN}_{\geq 0}})(\tilde{\bs \lambda}^{k}) \in \bs \lambda^k +  \beta_{\text d} ( \frac{1}{N}(2  C_{\text{f}} \tilde{\bs x}^{k} -  C_{\text{f}} \bs x^k- c_{\text{f}} ) - \L_m \bs \lambda^k)$, which leads to
\begin{multline} \label{eq:Alg2-L}
\textstyle
- \L_m \bs \lambda^k - \frac{1}{N} c_{\text{f}}
\in \nc_{\R^{mN}_{\geq 0}}(\tilde{\bs \lambda}^{k}) - \frac{1}{N} C_{\text{f}} \tilde{\bs x}^{k}   \\
\textstyle
- \frac{1}{N}C_{\text{f}}(\tilde{\bs x}^{k}-\bs x^{k}) +  \beta_{\text d} ^{-1}(\tilde{\bs \lambda}^{k}-\bs \lambda^{k} ).
\end{multline}
Let $\bs \omega^k := \col(\bs x^k,\bs \lambda^k)$, then the inclusions in \eqref{eq:Alg2-X}$-$\eqref{eq:Alg2-L} can be cast in compact form as 
\begin{align} \label{eq:Alg2-incl-comp}
\textstyle
-T_1(\bs \omega^k) \in T_2(\tilde{\bs \omega}^{k}) +\Phi (\tilde{\bs \omega}^{k}- \bs \omega^k),
\end{align}
where $T_1$, $T_2$ and $\Phi$ as in \eqref{eq:T1}, \eqref{eq:T2} and \eqref{eq:SandPhi}, respectively.
By making $\tilde{\bs \omega}^{k}$ explicit in \eqref{eq:Alg2-incl-comp}, we obtain
\begin{align}
\tilde{\bs \omega}^{k} = (\Id+\Phi^{-1} T_2)^{-1}\circ(\Id-\Phi^{-1} T_1)(\bs \omega^{k}),
\end{align}
which corresponds to $\tilde{\bs \omega}^{k} = R (\bs \omega^{k})$, where $R$ is the pFB operator in \eqref{eq:Rmap}. Finally, it follows by \eqref{eq:ALG2compact3}$-$\eqref{eq:ALG2compact4} that $\bs \omega^{k+1}= \bs \omega^k + \gamma^k(
R(\bs \omega^k)-\bs \omega^k)$, which concludes the proof.

\smallskip
(2): Next, we introduce some technical statements that we exploit later on in this proof.


\smallskip
\begin{lemma} \label{lem:PMa}
Let the step-sizes $\{\alpha_i,\beta_i\}_{i \in \mc I}$ satisfy Assumption \ref{ass:SS-choice1}. Then the following statements hold:
\begin{enumerate}[(i)]
\item $\Phi - \tau I \succeq 0$, with $\tau$ as in Assumption \ref{ass:SS-choice1},
\item $\|\Phi^{-1}\| \leq \tau^{-1} $.
{\hfill $\square$}
\end{enumerate}
\end{lemma}
\begin{proof} (i):
By the generalized Gershgorin circular theorem \cite[Th.~2]{feingold1962block}, each eigenvalue $\mu$ of the matrix $\Phi$ in \eqref{eq:SandPhi} satisfies at least one of the following inequalities:
\begin{align} \label{eq:GGCT1}
\mu & \geq \alpha_i^{-1} - \| C_{i}^\top\| , \quad  &\forall i \in \mc I,\\
\mu & \geq \textstyle \beta_i^{-1} - \frac{1}{N} \sum_{j = 1}^N \| C_{j}^\top\|, \quad  &\forall i \in \mc I. 
\label{eq:GGCT3}
\end{align}
Hence, if we set the step-sizes $\alpha_i,\beta_i$ as in Assumption \ref{ass:SS-choice1}, the inequalities \eqref{eq:GGCT1}-\eqref{eq:GGCT3} yield to $\mu \geq \tau $. It follows that the smallest eigenvalue of $\Phi$, i.e., $\mu_{\min}(\Phi)$, satisfies $\mu_{\min}(\Phi) \geq \tau>0$. Thus, $\Phi-\tau I$ is positive semi-definite.

(ii): Let $\mu_{\max}(\Phi)$ be the largest eigenvalue of $\Phi$. We have that $\mu_{\max}(\Phi)\ge \mu_{\min}(\Phi) \geq \tau$. Moreover, $\| \Phi \|=\mu_{\max}(\Phi) \ge \mu_{\min}(\Phi) = \frac{1}{\| \Phi^{-1} \|} \geq \tau$. Hence $\| \Phi^{-1} \| \leq \tau^{-1}$. 
\end{proof}

\smallskip
\begin{lemma} \label{lem:T12AV}
Let Assumptions \ref{ass:ConvCompSet} and \ref{ass:COCO} hold and the step-sizes $\{\alpha_i,\beta_i\}_{i \in \mc I}$ satisfy Assumption \ref{ass:SS-choice1}. The following properties hold in the $\Phi$-induced norm (i.e., $\|\cdot\|_\Phi$):
\begin{enumerate}[(i)]
\item $\Phi^{-1} T_1$ is $\delta \tau-$cocoercive and $(\Id-\Phi^{-1} T_1)$ is $\frac{1}{2 \delta \tau}-$averaged;
\item $\Phi^{-1} T_2$ is maximally monotone and $(\Id-\Phi^{-1} T_2)^{-1}$ is $\frac{1}{2}-$averaged.
{\hfill $\square$}
\end{enumerate}
\end{lemma}

\smallskip
\begin{proof}
(i): Since $T_1$ is single-valued and $\Phi^{-1}$ nonsingular, by Lemma \ref{lem:PMa} (i), for each $\bs \omega, \bs \omega' \in \bs \Omega \times \R^{nN}_{\geq 0}$
\begin{multline} \label{eq:cocoInQ} \textstyle
\|\Phi^{-1}T_1(\bs \omega)-\Phi^{-1}T_1(\bs \omega')\|^2_{\Phi}
= \|T_1(\bs \omega)-T_1(\bs \omega')\|^2_{\Phi^{-1}} \\
\textstyle
 \leq \|\Phi^{-1} \|\, \|T_1(\bs \omega)-T_1(\bs \omega') \|^2\\
\textstyle
\leq \frac{1}{\tau} \, \|T_1(\bs \omega)-T_1(\bs \omega') \|^2,
\end{multline}
where the last inequality follows by Lemma \ref{lem:PMa} (ii).
By \eqref{eq:cocoInQ} and the $\delta-$cocoercivity of $T_1$ (Lemma \ref{lem:RSM} (ii))
\begin{multline} \textstyle \label{eq:cocoInQ-2}
 \langle \Phi^{-1}T_1(\bs \omega) - \Phi^{-1}T_1(\bs \omega'),  \bs \omega - \bs \omega' \rangle_{\Phi}  = \\ 
 \textstyle  \langle T_1(\bs \omega) - T_1(\bs \omega'),  \bs \omega - \bs \omega' \rangle 
\ge  \delta \|T_1(\bs \omega)-T_1(\bs \omega') \|^2 \\
\textstyle
\ge  \delta  \tau \|\Phi^{-1} T_1(\bs \omega)-\Phi^{-1} T_1(\bs \omega') \|^2_{\Phi}.
\end{multline}
In other words, $\Phi^{-1} T_1$ is $\delta \tau-$cocoercive in the $\Phi-$induced norm. It follows from \cite[Prop. 4.33]{bauschke2017convex} that $(\Id-\Phi^{-1}T_1)$ is $\frac{1}{2\delta \tau}-$averaged in the $\Phi-$induced norm.

(ii):
$\Phi^{-1} T_2$ is maximally monotone in the $\Phi-$induced norm, since $T_2$ is maximally monotone by Lemma \ref{lem:RSM} (i). By \cite[Prop.~23.7]{bauschke2017convex}, the resolvent mapping $(\Id + \Phi^{-1}T_2)$ is $\frac{1}{2}-$averaged (or firmly-nonexpansive, see \cite[Remark~4.24]{bauschke2017convex}) in the $\Phi-$induced norm, since $\Phi^{-1}T_2$ is maximally monotone in the same norm.
\end{proof}

\smallskip
\begin{lemma}\label{lem:Raveraged}
Let Assumptions \ref{ass:ConvCompSet}, \ref{ass:COCO} hold and the step-sizes $\{\alpha_i,\beta_i\}_{i \in \mc I}$ satisfy Assumption \ref{ass:SS-choice1}. Then, the pFB operator $R=(\Id+\Phi^{-1} T_2)^{-1} \circ (\Id-\Phi^{-1} T_1)$ is 
$\nu-$averaged in the $\Phi-$induced norm (i.e., $\| \cdot \|_\Phi$), with $\nu := \frac{2\delta \tau}{4 \delta \tau-1} \in (\frac{1}{2},1)$.
{\hfill $\square$}
\end{lemma}
\begin{proof}
By \cite[Proposition 4.4]{bauschke2017convex} , the mapping $R$ is $\left( \frac{2\delta \tau}{4 \delta \tau-1} \right)-$averaged with respect to $\| \cdot\|_{\Phi}$, since composition of $(\Id+\Phi^{-1} T_2)^{-1}$ and $(\Id-\Phi^{-1} T_1)$ which are $\frac{1}{2}-$ and $\frac{1}{2\delta \tau}-$averaged in $\|\cdot \|_{\Phi}$, respectively, by Lemma \ref{lem:T12AV}. Moreover, $ \frac{2\delta \tau}{4 \delta \tau-1} \in (\frac{1}{2},1)$, since $ \tau >\frac{1}{2\delta}$, by Assumption \ref{ass:SS-choice1}.
\end{proof}

\medskip
The fixed-point iteration \eqref{eq:R}, that corresponds to Algorithm 2 by the first part of this proof, is the Krasnosel'skii-Mann iteration on the mapping $R$, which is $\nu-$averaged, with $\nu \in (\frac{1}{2},1)$, by Lemma \ref{lem:Raveraged}. The convergence of \eqref{eq:R} to some $\bs \omega^*:=\col(\bs x^*,\bs \lambda^*) \in \fix(R)$ follows by \cite[Prop. 5.15]{bauschke2017convex}. To conclude, we note that $\bs \omega^* \in \fix(R)= \zer(\Phi^{-1} T_1+\Phi^{-1}T_2)$, by \cite[Prop. 25.1 (iv)]{bauschke2017convex}, and that $ \zer(\Phi^{-1}T_1+\Phi^{-1}T_2)= \zer(T)$, with $T$ as in \eqref{eq:T}, since $\Phi\succ 0$, by Lemma \ref{lem:PMa} (i), and $T_1+T_2=T$. Since the limit point $\bs \omega^* \in \zer(T) \neq \varnothing$, by Proposition \ref{pr:EOC} (i), then $\bs x^*$ is a v-GNE of the game in \eqref{eq:Game}, by Proposition \ref{pr:EOC} (ii), thus concluding the proof.
{\hfill $\blacksquare$}

\subsection{Proof of Lemma \ref{lem:ESTeRAV}}
\label{proof:lem:ESTeRAV}
We prove equation (i) by induction. At step zero, $\bar{\sigma}^0 =
\bar{x}^0$ holds if the estimates are initialized as $\sigma_i^0 = x_i^0$, for all $i \in \mc I$. At step $k$, we assume that $\bar{\sigma}^k = \bar{x}^k$. To conclude the proof, we show that relation (i) holds at step $k+1$:
\begin{align*}
\bar{\sigma}^{k+1} &= \textstyle
\frac{1}{N}(\1^\top\otimes I_n)((W(k) \otimes I_n) \bs \sigma^k + \bs x^{k+1}- \bs x^k), \\ 
&= \textstyle
\frac{1}{N}(\1^\top\otimes I_n) (W(k) \otimes I_n) \bs \sigma^k + \bar{x}^{k+1} -\bar{x}^k,\\
&= \bar \sigma^k + \bar{x}^{k+1} -\bar{x}^k = \bar{x}^{k+1}.
\end{align*}
The first equality follows from the updating rule of the $\sigma_i$'s in Algorithm 3, the second follows by definition of $\bar{x}^k$, i.e., $\bar{x}^k = \frac{1}{N}(\1^\top \otimes I_n) \bs x^k$, the third follows since the mixing matrix $W(k)$ is column stochastic, i.e., $\1^\top W(k) = \1^\top$, by Assumption \ref{ass:DSMM}, while the last equality follows from the induction step $k$, i.e., $\bar{\sigma}^k = \bar{x}^k$.
The proof of equations (ii) and (iii) are analogous.

\subsection{Proof of Lemma \ref{lem:SurrToReal}}
\label{proof:SurrToReal}
For easy of notation, this proof is developed for the scalar case, i.e., $n=m=1$. In this case, we can write  $\| \hat{\bs \sigma}^k - \1 \otimes \bar x^k \| = \| (W(k) \otimes I_n) \bs \sigma^k - \1 \otimes \bar x^k    \| = \| W(k) \bs \sigma^k - \bar x^k   \1 \|$.

\smallskip
(i): The update of the estimates $\sigma_i$'s in Algorithm 3 can be written in a compact form as
\begin{align} \label{eq:sigComp}
\bs \sigma^{k+1} = W(k) \bs \sigma^k + \bs x^{k+1}- \bs x^k.
\end{align}
By telescoping \eqref{eq:sigComp}, we obtain
\begin{align} \nonumber
\bs \sigma^{k+1} &= W(k)(W(k-1)\bs \sigma^{k-1} + \bs x^{k} - \bs x^{k-1})\\
\nonumber
&\quad + \bs x^{k+1} - \bs x^k\\[.2em]
\nonumber
&  = \Psi(k,k-1)\bs \sigma^{k-1} + \Psi(k,k)(\bs x^{k} - \bs x^{k-1}) \\
\nonumber
&\quad + \bs x^{k+1} - \bs x^k
\\[.2em]
\nonumber
&= \cdots \\[.2em]
\nonumber
&\textstyle = \Psi(k,0)\bs \sigma^{0} + \sum_{s=1}^k \Psi(k,s)(\bs x^{s} - \bs x^{s-1})\\
&\quad + \bs x^{k+1} - \bs x^k,
\label{eq:sigTelescop}
\end{align}
\blue{where the transition matrices $\Psi(\cdot,\cdot)$'s are defined in \eqref{eq:TransMatr}. 
By rearranging \eqref{eq:sigComp}, we can write $ W(k) \bs \sigma^k = \bs \sigma^{k+1} - \bs x^{k+1}+ \bs x^k $. Then, by exploiting the equivalence in \eqref{eq:sigTelescop}, we have}
\begin{align} \textstyle \label{eq:sigTelescop_W}
W(k) \bs \sigma^k= \Psi(k,0)\bs \sigma^{0} + \sum_{s=1}^k \Psi(k,s)(\bs x^{s} - \bs x^{s-1}).
\end{align}
Now, consider $\bar \sigma^k$, which may be written as follows:
\begin{equation*} \textstyle
\bar \sigma^k = \bar \sigma^{k-1} + (\bar \sigma^{k}-\bar \sigma^{k-1})=	\bar{\sigma}^0 + \sum_{s=1}^k(\bar \sigma^{s}-\bar \sigma^{s-1}).
\end{equation*}
By Lemma \ref{lem:ESTeRAV}, we have that $\bar{\sigma}^s = \bar{x}^s$  $\forall s \geq 0 $, which leads to
\begin{align}
\textstyle
\nonumber
\bar x^k = \bar \sigma^k &
\textstyle
= \bar{\sigma}^0  + \sum_{s=1}^k (\bar{x}^s-\bar{x}^{s-1})\\
&\textstyle
= \frac{1}{N} \1^\top  \bs \sigma^0 + \sum_{s=1}^k \frac{1}{N} \1^\top (\bs{x}^s-\bs{x}^{s-1}).
\label{eq:RewriteS0}
\end{align}
From equations \eqref{eq:sigTelescop_W} and  \eqref{eq:RewriteS0}, we have the following:
\begin{align} 
\nonumber
\textstyle
& \| W(k) \bs \sigma^k - \bar x^k   \1 \| \textstyle  \\[.3em]
\nonumber
&  \textstyle  \qquad = \| (\Psi(k,0)-\frac{1}{N}\1 \1^\top ) \bs \sigma^{0}\\
\nonumber
\textstyle
&\textstyle  \qquad \quad + \sum_{s=1}^k (\Psi(k,s)-\frac{1}{N}\1 \1^\top )(\bs{x}^s-\bs{x}^{s-1})\|\\
\nonumber
& \textstyle
 \qquad \overset{(\text{a})}{\leq} \| \Psi(k,0)-\frac{1}{N}\1 \1^\top \| \| \bs \sigma^{0} \|
 \\
 \nonumber
 & \textstyle \qquad \quad 
+ \sum_{s=1}^k \| \Psi(k,s)-\frac{1}{N}\1 \1^\top \| \| \bs{x}^s-\bs{x}^{s-1}\|\\
& \qquad  \overset{(\text{b})}{\leq} \textstyle
\theta \rho^k \| \bs \sigma^{0} \| + \sum_{s=1}^k \theta \rho^{k-s}\| \bs{x}^s-\bs{x}^{s-1} \|,
\label{eq:TelescopUB}
\end{align}
\blue{where (a) follows from the Cauchy–Schwarz inequality, while
(b) since $\|\Psi(k,s) -\frac{1}{N} \1 \1^\top \| \leq \theta \rho^{k-s}$ for all $k \geq s \geq 0$, by Lemma \ref{lem:Tsi}.}
Next, we find an upper bound for $\|\bs{x}^s-\bs{x}^{s-1} \|$ in \eqref{eq:TelescopUB}. The update of the decisions $x_i$'s can be written in a compact form as $\bs x^{k+1} = \bs x^{k} + \gamma^k(\tilde{\bs x}^{k} - \bs x^{k})$.
We note that $\tilde {x}_i^k, x_i^k \in  \Omega_i$, for all $k \ge 0$ since $\tilde x_i^k$ is obtained by projecting onto $\Omega_i$ and $x_i^k = (1-\gamma^k)x_i^{k-1} + \gamma^k \tilde x_i^k $ is a convex combination of elements of the convex set $\Omega_i$.
Since all the sets $ \Omega_i$'s are compact, by Assumption \ref{ass:ConvCompSet}, it follows that for some constant $B_{\Omega}$, we have
\begin{align}
\label{eq:UBx}
\|\bs{x}^s-\bs{x}^{s-1} \| &= \gamma^{s-1} \| \tilde{\bs{x}}^{s-1}-\bs{x}^{s-1}\|  \leq \gamma^{s-1} B_{\Omega}.
\end{align}
By combining \eqref{eq:UBx} and \eqref{eq:TelescopUB}, we obtain
\begin{align} \nonumber \textstyle
\| W(k) \bs \sigma^k - \bar x^k   \1 \| &\textstyle \leq
\theta \rho^k B_{\Omega} + \sum_{s=1}^k \theta \rho^{k-s} \gamma^{s-1} B_{\Omega},
\end{align}
where we exploited the initialization step of Algorithm 3, i.e., $\bs \sigma^0 = \bs x^0 \in \bs \Omega$, from which $ \| \bs \sigma^{0} \| \leq B_{\Omega}$.

\smallskip
(ii): The update of the estimates $z_i$'s in Algorithm 3 can be written in a compact form as
\begin{align} \label{eq:lamComp}
\bs z^{k+1} = W(k) \bs z^k + \bs \lambda^{k+1}- \bs \lambda^k.
\end{align}
By telescoping \eqref{eq:lamComp}, we obtain
\begin{align} 
\nonumber
\textstyle
& \| W(k) \bs z^k - \bar \lambda^k   \1 \| \textstyle  \\
& \qquad  \leq \textstyle
\theta \rho^k \| \bs z^{0} \| + \sum_{s=1}^k \theta \rho^{k-s}\| \bs{\lambda}^s-\bs{\lambda}^{s-1} \|,
\label{eq:TelescopUB_lam}
\end{align}
\blue{To upper bound $\|\bs{\lambda}^s-\bs{\lambda}^{s-1} \| $, we note that the dual update in Alg. 3 reads in compact form as $\bs \lambda^{s} = \bs \lambda^{s-1} + \gamma^s (\tilde {\bs \lambda}^{s-1}-\bs \lambda^{s-1})$ and that the dual sequence $(\bs{\lambda}^s)_{s \in \bb N}$ is positive and $B_D-$norm bounded, by Assumption \ref{ass:BDV}. Hence, we have $\|\bs{\lambda}^s-\bs{\lambda}^{s-1} \| \leq \gamma^{s-1} B_D
$, that substituted into \eqref{eq:TelescopUB_lam} gives}
\begin{align} 
\nonumber
\textstyle
& \| W(k) \bs z^k - \bar \lambda^k   \1 \| \textstyle \leq \textstyle
\theta \rho^k B_D + \sum_{s=1}^k \theta \rho^{k-s} \gamma^{s-1} B_D.
\end{align}

(iii): The update of the estimates $y_i$'s in Algorithm 3 can be written in a compact form as
\begin{multline}
\bs y^{k+1} = W(k) \bs y^k + C_{\text d}(  2\tilde{\bs x}^k - \bs x^k)\\
\label{eq:yComp}
- C_{\text d}(  2 \tilde{\bs x}^{k-1} - \bs x^{k-1}),
\end{multline}
By telescoping \eqref{eq:yComp} \blue{(as explained in \eqref{eq:sigTelescop})}, we obtain
\begin{multline}
\textstyle
\bs y^{k+1} = \Psi(k,0)\bs y^{0} + \sum_{s=1}^k \Psi(k,s) \\
\textstyle
\quad \cdot \big( C_{\text d}(2\tilde{\bs x}^{s-1}  - \bs x^{s-1}) -C_{\text d}(2\tilde{\bs x}^{s-2} - \bs x^{s-2})\big) \\
+ C_{\text d}(2\tilde{\bs x}^{k} - \bs x^{k}) - C_{\text d}(2\tilde{\bs x}^{k-1} - \bs x^{k-1}).
\label{eq:Telescop_D}
\end{multline}
Now, consider $\bar y^k$, which may be written as follows:
\begin{equation*} \textstyle
\bar y^k =	\bar{y}^0 + \sum_{s=1}^{k}\{\bar y^{s-1}-\bar y^{s-2}\} + \bar y^{k} - \bar{y}^{k-1}.
\end{equation*}
By Lemma \ref{lem:ESTeRAV}, we have that $\bar{y}^s = \bar{d}^s = \frac{1}{N} \1^\top  C_{\text d}(2  \tilde{\bs x}^s - \bs x^s  )- c $, for all $s \geq 0 $, which leads to
\begin{align} \nonumber
\textstyle
\bar d^k = \bar y^k 
 & =  \textstyle \frac{1}{N} \1^\top  \bs y^0 +  \sum_{s=1}^k \frac{1}{N}  \1^\top \\ 
& \quad \textstyle \nonumber
\cdot \big( C_{\text d}(2  \tilde{\bs x}^{s-1} - \bs x^{s-1}  )- C_{\text d}(2  \tilde{\bs x}^{s-2} - \bs x^{s-2} )\big)\\
& \quad \textstyle
+ \frac{1}{N} \1^\top \big( C_{\text d}(2  \tilde{\bs x}^{k} - \bs x^{k}  )- C_{\text d}(2  \tilde{\bs x}^{k-1} - \bs x^{k-1} ) \big).
\label{eq:RewriteD0}
\end{align}
From the relations \eqref{eq:Telescop_D} and  \eqref{eq:RewriteD0}, we have the following:
\begin{align} 
\nonumber
\textstyle
& \| \bs y^{k+1} - \bar d^k   \1 \| \textstyle  \\
\nonumber
&  \textstyle  \quad \overset{(\text{a})}{=} \| (\Psi(k,0)-\frac{1}{N}\1 \1^\top ) \bs y^{0} + \sum_{s=1}^k (\Psi(k,s)-\frac{1}{N}\1 \1^\top )\\
&  \textstyle \nonumber \quad \quad \cdot
\big( C_{\text d}(2  \tilde{\bs x}^{s-1} - \bs x^{s-1}  )- C_{\text d}(2  \tilde{\bs x}^{s-2} - \bs x^{s-2} )\big) \\
&\nonumber \textstyle  \quad \quad +
(I_N-\frac{1}{N} \1 \1^\top)\\
&  \textstyle \nonumber \quad \quad \cdot 
\big( C_{\text d}(2  \tilde{\bs x}^{k} - \bs x^{k}  )- C_{\text d}(2  \tilde{\bs x}^{k-1} - \bs x^{k-1} )\big)
\|\\
\nonumber
& \textstyle
 \quad \overset{(\text{b})}{\leq} \| \Psi(k,0)-\frac{1}{N}\1 \1^\top \| \| \bs y^{0} \|
 \\
 \nonumber
 & \textstyle \quad \quad 
+ \sum_{s=1}^k \| \Psi(k,s)-\frac{1}{N}\1 \1^\top \| \|C_{\text d} \| \\
& \nonumber \textstyle \quad \quad \cdot
\| (2  \tilde{\bs x}^{s-1} - \bs x^{s-1}  )- (2  \tilde{\bs x}^{s-2} - \bs x^{s-2} )\|\\
\nonumber
& \textstyle \quad \quad +
 \|C_{\text d} \| \| (2  \tilde{\bs x}^{k} - \bs x^{k}  )- (2  \tilde{\bs x}^{k-1} - \bs x^{k-1} )\|\\
\nonumber
& \quad  \overset{(\text{c})}{\leq} \textstyle
\theta \rho^k B_Y + \sum_{s=1}^k \theta \rho^{k-s}\\
& \nonumber \textstyle \quad \quad \cdot
 \|C_{\text d} \| \| (2  \tilde{\bs x}^{s-1} - \bs x^{s-1}  )- (2  \tilde{\bs x}^{s-2} - \bs x^{s-2} )\|\\
& \textstyle \quad \quad +
  \|C_{\text d} \| \| (2  \tilde{\bs x}^{k} - \bs x^{k}  )- (2  \tilde{\bs x}^{k-1} - \bs x^{k-1} )\|,
\label{eq:TelescopUB_2}
\end{align}
\blue{where the first equality, (a), follows by substituting \eqref{eq:Telescop_D} and  \eqref{eq:RewriteD0} to $ \bs y^{k+1} $ and $\bar d^k$, respectively, (b) follows from Cauchy--Schwartz inequality and (c) since $\|\Psi(k,s) -\frac{1}{N} \1 \1^\top \| \leq \theta \rho^{k-s}$ for all $k \geq s \geq 0$, by Lemma \ref{lem:Tsi}.}
Now we build an upper bound for $\| (2  \tilde{\bs x}^{s} - \bs x^{s}  )- (2  \tilde{\bs x}^{s-1} - \bs x^{s-1} )\|$ in \eqref{eq:TelescopUB_2}:
\begin{align} \nonumber
\| (2  \tilde{\bs x}^{s} - \bs x^{s}  ) &- (2  \tilde{\bs x}^{s-1} - \bs x^{s-1} )\| \\
\nonumber
&\overset{(\text{a})}{\leq} 2\|\tilde{\bs x}^{s}- \tilde{\bs x}^{s-1} \| + \|\bs x^{s}  -\bs x^{s-1}   \|\\
& \overset{(\text{b})}{\leq}
2\|\tilde{\bs x}^{s}- \tilde{\bs x}^{s-1} \| + \gamma^{s-1}B_\Omega,
\label{eq:tildeDiff}
\end{align}
\blue{where (a) follows from the triangular inequality and (b) follows from \eqref{eq:UBx}
}
Next, we build an upper bound for the term $\|\tilde{\bs x}^{s}- \tilde{\bs x}^{s-1} \|$ in the right hand side of \eqref{eq:tildeDiff}.
\begin{align} \textstyle \nonumber
& \|  \tilde{\bs x}^{s}- \tilde{\bs x}^{s-1} \| \\
& \nonumber
\overset{(\text{a})}{=} \| 
\proj_{\bs \Omega}\big( \bs x^{s} -  \alpha_{\text d} (  \bs F(\bs x^{s},\hat{\bs \sigma}^{s} ) + C_{\text d}^\top \hat{\bs z}^{s} \big)  \\
\nonumber
& \qquad - \proj_{\bs \Omega}\big( \bs x^{s-1} -  \alpha_{\text d} (  \bs F(\bs x^{s-1},\hat{\bs \sigma}^{{s-1}} ) + C_{\text d}^\top \hat{\bs z}^{{s-1}} \big)
\|\\
& \overset{(\text{b})}{\leq} \textstyle \nonumber
\| \bs x^{s} - \bs x^{s-1} -\alpha_{\text d} \\
\nonumber
& \quad \textstyle \cdot \big(
\bs{F}(\bs x^{s}, W(s)\bs \sigma^{s}) - 
\bs{F}(\bs x^{s-1}, W(s\!-\!1)\bs \sigma^{s-1}) \\
\nonumber
& \quad
+ {C}_\text{d}^\top W(s) \bs z^{s} - {C}_\text{d}^\top W(s\!-\!1) \bs z^{s-1} \big) \| \\
\nonumber
&  \overset{(\text{c})}{\leq}
\| \bs x^{s} - \bs x^{s-1} \|  \\
\nonumber
& \quad 
+ L_{\bs{F}} \| {\alpha}_\text{d} \|   \,\|
\begin{bmatrix}
\bs x^{s} - \bs x^{s-1} \\
\nonumber
 W(s)\bs \sigma^{s} - W(s-1)\bs \sigma^{s-1}
\end{bmatrix}
\| \\
\nonumber
& \quad
+ \| \alpha_\text{d} \| \|C_\text{d} \|
 \|W(s) \bs z^{s} - W(s-1) \bs z^{s-1} \| \\
 \nonumber
 &  \overset{(\text{d})}{\leq}
 (1+L_{\bs{F}} \| \alpha_{\text d} \|)  \|\bs x^{s}  -  \bs x^{s-1} \|\\
 \nonumber
 &\quad + L_{\bs{F}} \|\alpha_{\text d} \|  \|W(s)\bs \sigma^{s} - W(s\!-\!1)\bs \sigma^{s-1} \| \\
 &\quad + \| \alpha_{\text d} \| \|C_\text{d} \|
 \|W(s) \bs z^{s} - W(s\!-\!1) \bs z^{s-1} \|,
 \label{eq:W1W2}
\end{align}
\blue{
where (a) follows by exploiting the compact update of $\tilde{\bs x}^s$ in \eqref{eq:ALG3compactX_alg3}, (b) follows by the nonexpansiveness of the projection operator, (c) follows by exploiting, in sequence, the triangular inequality, the Lipschitz continuity of $\bs F$ (Assumption \ref{ass:LC-TF}), and the Cauchy-Schwartz inequality, finally (d) follows from the relation $\|
\left[
\begin{smallmatrix}
a \\ b
\end{smallmatrix}
\right]
\|
=
\sqrt{\|a\|^2 + \|b \|^2} \leq \|a \|+\| b\|$.
}
Now, we find an upper bound the last two terms in \eqref{eq:W1W2}.
\begin{align} \nonumber
& \|W(s)\bs \sigma^{s} - W(s\!-\!1)\bs \sigma^{s-1} \| \\
\nonumber
& \overset{(\text{a})}{=} \|W(s)\bs \sigma^{s} -  \bs \sigma^{s} + \bs x^{s} - \bs x^{s-1} \| \\
\nonumber
& \overset{(\text{b})}{\leq} \| W(s)\bs \sigma^{s} -  \bs \sigma^{s} \| + \| \bs x^{s} - \bs x^{s-1}  \|\\
\nonumber
& \overset{(\text{c})}{=} \| W(s)\bs \sigma^{s} -  \1 \bar x^{s}  - (  \bs \sigma^{s} -  \1 \bar x^{s}) \|\\
\nonumber
& \quad + \| \bs x^{s} - \bs x^{s-1}  \|\\
\nonumber
& \overset{(\text{d})}{\leq} \| W(s)\bs \sigma^{s} - \1 \bar x^{s} \| + \| \bs \sigma^{s} - \1 \bar x^{s}  \|\\
\nonumber
& \quad  + \gamma^{s-1} B_{\Omega}\\
\nonumber
& \textstyle \overset{(\text{e})}{\leq} \theta B_{\Omega} \rho^{s}  + \theta B_{\Omega} \sum_{\ell=1}^{s}  \rho^{s-\ell} \gamma^{\ell-1}\\
\nonumber
&  \textstyle  \quad + \theta B_{\Omega} \rho^{s-1}  + \theta B_{\Omega} \sum_{\ell=1}^{s-1}  \rho^{(s-1)-\ell} \gamma^{\ell-1} +  \gamma^{s-1} B_\Omega\\
\nonumber
& \quad + \gamma^{s-1} B_{\Omega}\\
& \textstyle \overset{(\text{f})}{\leq} 2 \left(\theta B_{\Omega} \rho^{s-1} \right) + 4 \left( \theta B_{\Omega} \rho^{-1} \sum_{\ell=1}^{s}  \rho^{s-\ell} \gamma^{\ell-1} \right),
\label{eq:WsigUB}
\end{align}
\blue{
where (a) follows since $W(s\!-\!1)\bs \sigma^{s-1}= \bs \sigma^{s}- \bs x^{s}+\bs x^{s-1}$ by \eqref{eq:sigComp}, (b) from the triangular inequality, (c)  by summing and subtracting $ \bar{x}^{s}\1 $ within the fist term, (d) by the triangular inequality and substituting to $\| \bs x^{s} - \bs x^{s-1}  \|$ the bound in \eqref{eq:UBx}, (e) by substituting to $\| W(s)\bs \sigma^{s} - \1 \bar x^{s} \|$ the upper bound derived in Lemma \ref{lem:SurrToReal} (i) and to $\| \bs \sigma^{s} - \1 \bar x^{s} \|$ a bound similarly derived, (f) follows by noticing that
$\rho^{s} < \rho^{s-1}$, since $0 < \rho < 1$ by Assumption \ref{ass:DSMM}. 
}
Similarly, for the last addend in \eqref{eq:W1W2}, we can derive the following bound:
\begin{multline}
\|W(s)\bs z^{s} - W(s\!-\!1)\bs z^{s-1} \|  \\
\textstyle
\leq 2 \theta B_{D} \rho^{s-1}  + 4 \theta B_{D}\rho^{-1} \sum_{\ell=1}^{s}  \rho^{s-\ell} \gamma^{\ell-1}.
 \label{eq:WzUB}
\end{multline}
Finally, by combining \eqref{eq:W1W2} with \eqref{eq:WsigUB} and \eqref{eq:WzUB}, we obtain an upper bound for $\|\tilde{\bs x}^{s}- \tilde{\bs x}^{s-1} \|$, i.e.,
\begin{align} \nonumber
& \| \tilde{\bs x}^{s}- \tilde{\bs x}^{s-1} \| \\
&
\leq 
 \nonumber \textstyle
\underbrace{\|  \alpha_{\text d} \| 2 \theta (L_{\bs{F}} B_{\Omega} + \|C_{\text d}\|B_D)}_{:=\epsilon_1} \rho^{s-1}  \\
& \quad \textstyle \nonumber + \underbrace{  4 \theta \rho^{-1} \| \alpha_{\text d} \| (L_{\bs{F}} B_{\Omega} + \|C_{\text d}\|B_D)}_{:=\epsilon_2} \sum_{\ell=1}^{s}  \rho^{s-\ell} \gamma^{\ell-1}\\
& \nonumber \quad +
\underbrace{(B_\Omega +    \| \alpha_{\text d} \|  L_{\bs F} B_\Omega)}_{:=\epsilon_3} \gamma^{s-1}\\
& \leq 
\textstyle
\epsilon_1 \rho^{s-1}  +(\epsilon_2 + \epsilon_3) \sum_{\ell=1}^{s}  \rho^{s-\ell} \gamma^{\ell-1}.
\label{finUBTild}
\end{align}
Now, by substituting \eqref{finUBTild} into \eqref{eq:tildeDiff}, we obtain
\begin{align} \textstyle \nonumber
&\| (2  \tilde{\bs x}^{s} - \bs x^{s}  ) - (2  \tilde{\bs x}^{s-1} - \bs x^{s-1} )\| 
 \\
 \nonumber
&\textstyle \leq 2 \epsilon_1 \rho^{s-2} +2(\epsilon_2 + \epsilon_3) \sum_{\ell=1}^{s}  \rho^{s-\ell} \gamma^{\ell-1} + \gamma^{s-1} B_\Omega\\
\nonumber
& \leq \textstyle  \underbrace{2 \epsilon_1}_{:=\delta_1} \rho^{s-1} +\underbrace{(2 \epsilon_2 + 2 \epsilon_3 + B_\Omega)}_{:=\delta_2} \sum_{\ell=1}^{s}  \rho^{s-\ell} \gamma^{\ell-1} \\
& =: \phi^{s},
\label{eq:phiSeq}
\end{align}
where $(\phi^s)_{s \in \bb N}$ is the scalar  (vanishing) sequence in \eqref{eq:auxSeq}, with $\rho$ is as in \eqref{eq:rho} and $(\gamma^k)_{k \in \bb N}$ as in Assumption \ref{ass:van-ss}.
Finally, by combining \eqref{eq:phiSeq} and \eqref{eq:TelescopUB_2}, we obtain the upper bound in Lemma \ref{lem:SurrToReal} (iii).
{\hfill $\blacksquare$}

\subsection{Proof of Lemma \ref{lem:BoundErr}}
\label{proof:lem_BoundErr}

From $\|
\left[
\begin{smallmatrix}
a \\ b
\end{smallmatrix}
\right]
\|
=
\sqrt{\|a\|^2 + \|b \|^2} \leq \|a \|+\| b\|$, it follows that 
\begin{align} \nonumber
\| e^k \| & = \| \col(\tilde{\bs x}^k, \tilde{\bs \lambda}^k) - \col(\tilde{\bs x}^k_{\text{A2}}, \tilde{\bs \lambda}^k_{\text{A2}}) \| \\
& \leq \|\tilde{\bs x}^k - \tilde{\bs x}^k_{\text{A2}}\| + \|\tilde{\bs \lambda}^k- \tilde{\bs \lambda}^k_{\text{A2}}\|.
\label{eq:plbounderr}
\end{align}

Next, we upper bound $\|\tilde{\bs x}^k - \tilde{\bs x}^k_{\text{A2}}\|$, where $\tilde{\bs x}^k$ and $\tilde{\bs x}^k_{\text{A2}}$ are defined in \eqref{eq:ALG3compactX_alg3} and \eqref{eq:CFalg2_x}, respectively.
\begin{align} \nonumber
& \|\tilde{\bs x}^k - \tilde{\bs x}^k_{\text{A2}}\|\\[.3em]
\nonumber
& \nonumber
= \| 
\proj_{\bs \Omega} \big( \bs x^{k} -  \alpha_{\text d} (  \bs F(\bs x^{k},\hat{\bs \sigma}^{k} ) + C_{\text d}^\top \hat{\bs z}^{k} \big)  \\
\nonumber
&   \qquad 
- \proj_{\bs \Omega} \big( \bs x^{k} -  \alpha_{\text d} ( \bs F( \bs x^{k},\bar x^{k} \1 ) + C_{\text d}^\top \bar \lambda^{k} \big)
\| \\
\nonumber
& \overset{(\text{a})}{\leq} \|\alpha_{\text d} \| \, 
\| \bs{F}(\bs x^k, \hat{\bs \sigma}^k) - \bs{F}(\bs x^k, \bar{\bs x}^k)
+ C_d^\top (\hat{\bs z}^k - \hat{\bs \lambda}^k) \| \\
\nonumber
& \overset{(\text{b})}{\leq}  L_{\bs{F}}  \| \alpha_{\text d} \| \|(W(k)\otimes I_n) \bs \sigma^k -  \1 \otimes \bar{x}^k \| \\
& \quad  + \| \alpha_{\text d} \| \|C_d \| \|(W(k)\otimes I_m) \bs z^k -  \1 \otimes \bar{\lambda}^k \|,
\label{eq:bXtilde}
\end{align}
\blue{ where the first inequality (a) follows by the nonexpansivity of the projection operator, and (b) follows by the triangular inequality and the Lipschitz continuity of $\bs{F}$ (Assumption \ref{ass:LC-TF}).}

Now, consider $\|\tilde{\bs \lambda}^k- \tilde{\bs \lambda}^k_{\text{A2}}\|$, where $\tilde{\bs \lambda}^k$ and $\tilde{\bs \lambda}^k_{\text{A2}}$ are defined in \eqref{eq:ALG3compact3_lam} and \eqref{eq:CFalg2_lam}, respectively. By exploiting the nonexpansiveness of the projection operator, we have
\begin{align} \nonumber
\|\tilde{\bs \lambda}^k - \tilde{\bs \lambda}^k_{\text{A2}}\|
\nonumber
& \leq 
\| \beta_{\text d} \|  \|(W(k)\otimes I_m) \bs z^k -  \1 \otimes \bar{\lambda}^k \| \\
& \quad + \| \beta_{\text d} \| \| \bs y^{k+1}-  \1 \otimes \bar{d}^k \|  .
\label{eq:bLambdaTilde}
\end{align}

Finally, by combining \eqref{eq:bLambdaTilde} and \eqref{eq:bXtilde} with \eqref{eq:plbounderr} we obtain the upper bound in Lemma \ref{lem:BoundErr}.
{\hfill $\blacksquare$}

\subsection{Proof of Lemma \ref{lem:summ-Err}}
\label{proof:lem:summ-Err}
By substituting the bounds on the estimation errors of Lemma \ref{lem:SurrToReal} into
the error bound in Lemma \ref{lem:BoundErr}, we obtain
\begin{align} \textstyle \nonumber
\gamma^k \| e^k \| & \leq 
 a_1 \underbrace{\gamma^k \rho^k}_{\text{Term 1}}  + a_2 \underbrace{\gamma^k \sum_{s=1}^k \rho^{k-s} \gamma^{s-1}}_{\text{Term 2}} \\
&  \quad + a_3 \underbrace{\gamma^k \phi^k}_{\text{Term 3}} + a_4 \underbrace{\gamma^k \sum_{s=1}^k \rho^{k-s} \phi^{s-1}}_{\text{Term} 4} ,
\label{eq:gamEk}
\end{align}
where $a_1$, $a_2$, $a_3$ and $a_4$ are positive constants defined as $a_1 := \theta B_\Omega (\|\alpha_{\text d} \| L_{\bs{F}} + \| \beta_{\text d} \|) + \theta B_D (\| \alpha_{\text d}\| \|C_d \| + \| \beta_{\text d}\|)$,
$a_2 := \theta B_{\Omega} \| \alpha_{\text d}\| L_{\bs{F}}  + \theta B_D (\| \alpha_{\text d}\| \|C_d \| + \| \beta_{\text d}\|)$, $a_3 := \|\beta_{\text d}\| \|C_d \|$ and $a_4 := \|\beta_{\text d} \|$.
Now, we show that each term on the right-hand side of \eqref{eq:gamEk} is summable, hence also the sequence $(\gamma^k \| e^k \|)_{k \in \bb N}$ is such, i.e., $\sum_{k=0}^{\infty} \gamma^k \| e^k \| < \infty$.

\smallskip
\textrm{Term 1}: To establish the convergence of $\sum_{k=0}^\infty \gamma^k \rho^k $, we note that $\gamma^k \leq \gamma^0$, for all $ k \in \bb N$, by Assumption \ref{ass:van-ss}, implying that $\sum_{k=0}^\infty \gamma^k \rho^k \leq \gamma^0 \sum_{k=0}^\infty \rho^k < \infty$, since $0 < \rho < 1$ by Lemma \ref{lem:Tsi}.

\smallskip
Term 2: Since $\gamma^k \le \gamma^{s-1}$, for all $k \ge s\!-\!1$ (Assumption \ref{ass:van-ss}), the following relations hold for the second term in the right-hand side of \eqref{eq:gamEk}:
\begin{align*} 
\sum_{k=0}^\infty  \gamma^k \left( \sum_{s = 1}^k \rho^{k-s } \gamma^{s-1} \right) & = \sum_{k=0}^\infty  \sum_{s = 1}^k \rho^{k-s } \gamma^k \gamma^{s-1} \\
\textstyle
&  \leq \sum_{k=0}^\infty \sum_{s = 1}^k \rho^{k-s } (\gamma^{s-1})^2.
\end{align*}
It follows by Lemma \ref{lem:ram} (b) that $\sum_{k=0}^\infty \sum_{s = 1}^k \rho^{k-s } (\gamma^{s-1})^2 < \infty$, since $\sum_{k=0}^\infty (\gamma^k)^2 < \infty$, $ \gamma^k \geq 0$ for all $k$ (Assumption \ref{ass:van-ss}) and $0 <\rho<1$.

\smallskip
Term 3: By exploting the definition of the sequence $(\phi^k)_{k \in \bb N}$ in Lemma \ref{lem:SurrToReal}, we can write
\begin{align*}
\sum_{k=0}^\infty \gamma^k \phi^k &= 
\sum_{k=0}^\infty \gamma^k \left( \delta_1 \rho^{k-1} + \delta_2 \sum_{\ell=1}^k \rho^{k-\ell} \gamma^{\ell - 1} \right) \\
& = \delta_1  \sum_{k=0}^\infty \gamma^k  \rho^{k-1}  + \delta_2  \sum_{k=0}^\infty \gamma^k \sum_{\ell=1}^k \rho^{k-\ell} \gamma^{\ell - 1}\\
& \leq \delta_1 \gamma^0 \sum_{k=0}^\infty  \rho^{k-1}  + \delta_2  \sum_{k=0}^\infty \sum_{\ell=1}^k \rho^{k-\ell} (\gamma^{\ell - 1})^2
\end{align*}
By exploiting the same technical reasoning in (i) and (ii), we can show that each term on the right-hand side of the previous inequality globally converges. Therefore, we conclude that $\sum_{k=0}^\infty \gamma^k \phi^k < \infty$.
\smallskip

Term 4: Since $\gamma^k \le \gamma^s$, for all $k \ge s$ (Assumption \ref{ass:van-ss}), the following hold for the last term in the right-hand side of \eqref{eq:gamEk}:
\begin{align*} 
\sum_{k=0}^\infty \gamma^k \left( \sum_{s = 1}^k \rho^{k-s } \phi^{s-1} \right) & = \sum_{k=0}^\infty \sum_{s = 1}^k \rho^{k-s } \gamma^k \phi^{s-1} \\
\textstyle
&  \leq  \sum_{k=0}^\infty \sum_{s = 1}^k \rho^{k-s } (\gamma^{s-1} \phi^{s-1}).
\end{align*}
It follows by Lemma \ref{lem:ram} (b) that $\sum_{k=0}^\infty \sum_{s = 1}^k \rho^{k-s } \gamma^{s-1} \phi^{s-1} < \infty$, since $\sum_{k=0}^\infty \gamma^{k} \phi^{k} < \infty$ by (iii), and $0 <\rho<1$.

\smallskip
\noindent
To conclude, since all the terms in the right-hand side of \eqref{eq:gamEk} are summable, then we have $\sum_{k=0}^\infty \gamma^k \| e^k \| < \infty$.
{\hfill $\blacksquare$}

\balance
\bibliographystyle{IEEEtran}
\bibliography{IEEEfull,library}

\end{document}